\documentclass[11pt]{amsart}
\usepackage{amssymb}
\usepackage{psfrag}
\usepackage{enumerate}
\usepackage{amsmath}
\usepackage[all]{xy}
\usepackage[latin1]{inputenc}

\usepackage{color}

\newtheorem{theorem}{Theorem}
\newtheorem{proposition}[theorem]{Proposition}
\newtheorem{lemma}[theorem]{Lemma}
\newtheorem{corollary}[theorem]{Corollary}

\theoremstyle{definition}
\newtheorem{definition}[theorem]{Definition}
\newtheorem{remark}[theorem]{Remark}
\newtheorem{example}[theorem]{Example}

\newcommand{\LL}{{\mathbb L}}

\newcommand{\RR}{{\mathbb R}}
\newcommand{\SSS}{{\mathbb S}}

\newcommand{\cU}{{\mathcal U}}
\newcommand{\cL}{{\mathcal L}}

\begin{document}

\title{Symplectic and Poisson Geometry on $b$-Manifolds}

\author{Victor Guillemin}\address{ Victor Guillemin, Department of Mathematics, Massachussets Institute of Technology, Cambridge MA, US, \it{e-mail: vwg@math.mit.edu}}
\author{Eva Miranda}\address{ Eva Miranda,
Departament de Matem\`{a}tica Aplicada I, Universitat Polit\`{e}cnica de Catalunya, Barcelona, Spain, \it{e-mail: eva.miranda@upc.edu}}
\author{Ana Rita Pires} \address{Ana Rita Pires, Department of Mathematics, Cornell University \it{e-mail: apires@math.cornell.edu}}
   \thanks{Eva Miranda is partially supported by the  Ministerio de Econom�a y Competitividad project GEOMETRIA ALGEBRAICA, SIMPLECTICA, ARITMETICA Y APLICACIONES with reference: MTM2012-38122-C03-01 and by the ESF network CAST. Ana Rita Pires had the partial support of an AMS-Simons Travel Grant.}

\date{\today}

\begin{abstract}
Let $M^{2n}$ be a Poisson manifold with Poisson bivector field $\Pi$. We say that $M$ is $b$-Poisson if the map $\Pi^n:M\to\Lambda^{2n}(TM)$ intersects the zero section transversally on a codimension one submanifold $Z\subset M$. This paper will be a systematic investigation of such Poisson manifolds. In particular, we will study in detail the structure of $(M,\Pi)$ in the neighborhood of $Z$ and using symplectic techniques define topological invariants which determine the structure up to isomorphism. We also investigate a variant of de Rham theory for these manifolds and its connection with Poisson cohomology.
\end{abstract}

\maketitle

\section{Introduction}

In her 2002 paper \cite{Radko}, Radko gave a classification of stable Poisson structures on a compact 2-manifold $M$, where stable for a Poisson bivector field $\Pi$ means that the map
$$\Pi:M\rightarrow \Lambda^2(TM)$$
is transverse to the zero section. In this paper we describe some partial generalizations of this result to higher dimensions, with the stability condition replaced by a more complicated condition which we call $b$-Poisson: the map
$$\Pi^n=\Pi\wedge\ldots\wedge\Pi:M\rightarrow\Lambda^{2n}(TM)$$
must be transverse to the zero section of $\Lambda^{2n}(TM)$. We denote by $Z$ the hypersurface in $M$ where this map is zero. Then, $\Pi$ restricted to $Z$ defines a regular Poisson structure $\Pi_Z$ on $Z$ with codimension-one symplectic leaves.

When generalizing the results in \cite{Radko} to $n$ dimensions, one is confronted with the question of whether or not a such a Poisson structure on $Z$ can be extended to a $b$-Poisson structure on a neighborhood of $Z$ in $M$. We use modular geometry techniques of Weinstein to give necessary and sufficient conditions on $\Pi_Z$ for this to be the case. Then, we address the more complicated issue of how many extensions exist up to Poisson isomorphism and show that these are classified by the elements of $H^1_\text{Poisson}(Z)$. To prove this we give an alternative definition of $b$-Poisson: we show following Melrose and Nest-Tsygan that a $b$-Poisson manifold can be regarded dually as a $b$-symplectic manifold. More explicitly, we define a $b$-manifold to be a pair $(M,Z)$, where $M$ is a manifold and $Z$ a (not necessarily connected) hypersurface in $M$. We define a $b$-category in which the objects are such pairs and the morphisms are maps
$$f:(M_1,Z_1)\rightarrow(M_2,Z_2)$$
where $f$ is transverse to $Z_2$ and $Z_1=f^{-1}(Z_2)$; we call these $b$-maps. If one defines the $b$-tangent bundle $^bTM$ and $b$-cotangent bundle $^bT^*M$ for $(M,Z)$ following \cite{Melrose}\footnote{As is described in section \ref{section:whatdoesbstandfor}, we have borrowed the name of $b$-manifolds from the \emph{b-calculus} developed by Melrose \cite{Melrose}. Another possible denomination, that of \emph{log-symplectic manifolds},  is inspired from Algebraic Geometry and the approach of Goto \cite{goto} and, more recently, by Gualtieri and Li in \cite{Gualtierili} in the holomorphic case.  Our approach to $b$-geometry in this paper is close to but not exactly the same as that of Melrose, where a $b$-manifold is a manifold with boundary and $Z=\partial M$.}, these have nice functorial properties with respect to $b$-maps. We then define a $b$-symplectic form to be a nondegenerate closed $b$-two form $\omega$, i.e., a section of $\Lambda^2(\,^bT^*M)$ which is nondegenerate at all $p\in Z$ and show that $\omega|_Z$ defines a Poisson vector field $v$ on $(Z,\Pi_Z)$ and hence a cohomology class in the quotient
$$H^1_\text{Poisson}(Z)=\frac{\text{Poisson vector fields}}{\text{Hamiltonian vector fields}}.$$

Our proof that the converse assertion is true, i.e., that $\left[v\right]$ determines up to isomorphism the $b$-Poisson structure on $M$ (locally in a neighborhood of $Z$), is then reduced to a standard Moser type argument: if $\omega_0$ and $\omega_1$ are two $b$-symplectic forms on a tubular neighborhood $\cU$ of $Z$ in $M$ and $\omega_0|_Z=\omega_1|_Z$, then $\omega_0-\omega_1=d\mu$ for some $b$-one form $\mu\in\,^b\Omega^1(\cU)$, and one gets a diffeomorphism on $\cU$ mapping $\omega_0$ to $\omega_1$ by integrating the $b$-vector field $v_t$ defined by
$$\iota_{v_t}\omega_t=\mu, \text{ where } \omega_t=(1-t)\omega_0+t\omega_1.$$

We also use this Moser technique to prove that for $\omega_0$ and $\omega_1$ globally defined on a compact $M$, $\omega_0$ and $\omega_1$ are globally $b$-symplectomorphic if the usual global Moser conditions are satisfied, i.e., if $\omega_t$ is $b$-symplectic and $\left[\omega_t\right]\in\,^bH^2(M,\RR)$ is independent of $t$, in addition to $\omega_t|_Z$ being independent of $t$. We will show that in dimension two this argument gives an alternative $b$-symplectic proof of Radko's theorem.

The main part of this paper (sections \ref{section:whatdoesbstandfor} to \ref{section:theextensionproblem}) is a more detailed account of these results.
In sections \ref{section:whatdoesbstandfor} and \ref{section:bcategory} we review and introduce a number of basic definitions in $b$-geometry, from $b$-manifolds to $b$-differential forms.
In section \ref{section:twopointsofview} we examine the notion of $b$-symplectic from the Poisson-geometer's perspective, prove that $b$-symplectic and $b$-Poisson are equivalent notions and investigate some consequences of this fact: these include the relation between Weinstein's splitting theorem and a $b$-version of Moser's theorem, as well as the role of Weinstein's modular vector field  in this theory.
In section \ref{section:cohomology}, we discuss Mazzeo-Melrose's $b$-analogue of the de Rham theorem, obtain some cohomolgy results for $b$-symplectic manifolds and compare it to two versions of Poisson cohomology.
In section \ref{section:normalforms} we discuss Darboux and Moser theorems for $b$-symplectic manifolds and use these $b$-methods to prove Radko's theorem, in section \ref{section:invariantsassociatedtoabpoissonstructure} we revisit the modular invariants introduced in \cite{guimipi}, and in section \ref{section:theextensionproblem} show that these $b$-methods enable us to generalize Radko's theorem to arbitrary dimension, by addressing the extension problem that we alluded to above: does a regular Poisson structure on $Z$ extend to a $b$-symplectic structure on $M$, and if so in how many different ways?

All these results seem to put the $b$-symplectic category closer to the symplectic world than to the usually cumbersome Poisson world\footnote{In this sense, it is interesting to point out that an  $h$-principle holds  also for $b$-symplectic manifolds as it has been proved by Freijlich \cite{freijlich}.}.

 Section 9 of this paper is devoted to issues that we will investigate in more detail in the future. In section \ref{section:integrability} we show that $b$-symplectic manifolds are integrable, i.e., that the Lie algebroid structure on $M$ defined by $\Pi$ can always be integrated to a symplectic groupoid structure. This turns out to be an easy corollary of a much more general result \cite{marui2} of Crainic-Fernandes, in a future paper we will describe in more detail examples of symplectic groupoids whose existence is a consequence of this result. We also give an explicit model for the symplectic groupoid integrating the exceptional hypersurface $Z$.

Finally, in section \ref{section:integrablesystems}, we discuss completely integrable systems on $b$-symplectic manifolds, i.e., systems $f_1,\ldots,f_n$ of $b$-Poisson commuting functions which are almost everywhere functionally independent. In dimension two, $f$ can be chosen to be nonsingular at points of $Z$ and in fact to be a defining function for $Z$. In dimension four, $f_1$ can be chosen generically to be a defining function for $Z$ and $f_2$ to have standard elliptic and hyperbolic singularities on the (2-dimensional) symplectic leaves of $Z$. Analogous results can be obtained for higher dimensions.
In  \cite{gmps}  we  study  Hamiltonian actions  and integrable systems in the $b$-symplectic context.\footnote{ Since the first version of this paper was posted, other authors got interested in the b-symplectic category and made progress on some interesting and related aspects studied in this paper. We would like to point out here the paper of Geoffrey Scott \cite{scott} about the geometry of  manifolds  with \emph{higher order singularities} called $b^k$-symplectic manifolds. Also many improvements in the study of the topology of the $b$-symplectic manifolds have been carried out in \cite{marcutosorno1}, \cite{cavalcanti} and \cite{frejlichmartinezmiranda}.}

\noindent {\bf{Acknowledgements:}} We have benefited from interesting, useful and lively discussions with Marius Crainic, Rui Loja Fernandes, Pedro Frejlich, Marco Gualtieri, David Martinez Torres, Ryszard Nest, Francisco Presas and Geoffrey Scott. These discussions have enriched and improved this paper.  We are particularly thankful to Geoffrey Scott for carefully reading a first version of this paper and suggesting some important amendments. We deeply thank Ionut Marcut and Boris Osorno for pointing out an error in the statement and proof of Theorem \ref{thm:poissoncohomology} in a previous version of this paper and sending us a copy of their paper \cite{marcutosorno2}. The proof of it is a combination of our Theorem \ref{thm:liealgebroid} and a key lemma in \cite{marcutosorno2}. This theorem can also be found in  \cite{marcutosorno2}.

\section{What does $b$ stand for?}\label{section:whatdoesbstandfor}

We recall the notions of $b$-tangent and $b$-cotangent bundles introduced by Richard Melrose in \cite{Melrose} as a framework to study differential calculus and differential operators on manifolds with boundary, and further studied by Ryszard Nest and Boris Tsygan in \cite{NestandTsygan} in the context of formal deformations of symplectic manifolds with boundary.

Let $(M,\partial M)$ be a manifold with boundary. A neighborhood of a point $p\in\partial M$ is diffeomorphic to the upper half $n$-space
$$H_{+}^n=\{(x_1,\dots, x_n)| x_1\geq 0\},$$
and locally any vector field tangent to $\partial M$ {can be written as a linear combination of $x_1\frac{\partial}{\partial x_1}, \frac{\partial}{\partial x_2}, \ldots, \frac{\partial}{\partial x_n}$ with smooth coefficients. More formally, we can use sheaves to define the tangent and cotangent bundles: let $U\subset M$ be an open set and denote by $\Gamma(U,^b \mathcal{T})$ the set of vector fields on $U$ which are tangent to $\partial M$, then by varying the open set $U$ we obtain a sheaf on $M$. For a boundary point  $p$ this set is generated by $\langle x_1{\frac{\partial}{\partial x_1}}_p,{\frac{\partial}{\partial x_2}}_p\dots {\frac{\partial}{\partial x_n}}_p \rangle$ and therefore, for open sets $U\subset M$ diffeomorphic to $H_{+}^n$, the set $\Gamma(U,^b \mathcal{T})$ is freely generated over $C^{\infty}(U)$ by the sections $x_1{\frac{\partial}{\partial x_1}}, {\frac{\partial}{\partial x_2}},\dots,{\frac{\partial}{\partial x_n}}$. According to the Serre-Swan theorem \cite{swan}, there exists a vector bundle having $\Gamma(U,^b \mathcal{T})$ as local sections.  }

\begin{definition}
The \textbf{$b$-tangent bundle} $^b T(M)$ of a manifold with boundary $(M,\partial M)$ is the vector bundle associated to the locally free sheaf $\Gamma(U,^b \mathcal{T})$. The \textbf{$b$-cotangent bundle} $^b T^*(M)$ is the dual bundle to $^b T(M)$.
\end{definition}

Note that the $b$-cotangent bundle $^b T^*(M)$ is locally generated by the sections: $\frac{dx_1}{x_1},dx_2\dots, dx_n$.

The manifolds which we study in this paper are not manifolds with boundary, but rather manifolds with a distinguished hypersurface. As we will see in section \ref{section:bcategory}, the appropriate notions of tangent and cotangent bundles are basically the ones above\footnote{This was pointed out in \cite{bookCannasWeinstein} \S 17.4 in the context of giving an example of a Lie algebroid associated to the bundle of vectors tangent to $M$ which are furthermore tangent to a fixed hypersurface of $M$.}. For this reason, we adopt the notation $\,^bTM$ and $\,^bT^*M$ and furthermore name our objects of study ``$b$-manifolds''.

Some results proved for manifolds with boundary hold under appropriate translation for our $b$-manifolds, for example the Darboux theorem for closed nondegenerate 2 forms \cite{NestandTsygan} and the Mazzeo-Melrose splitting of the cohomology groups of $(M,\partial M)$ in terms of the cohomology of $M$ and of $\partial M$ \cite{Melrose}, as we will see in the next sections.

\section{Differential forms on $b$-manifolds}\label{section:bcategory}

In this section we introduce $b$-manifolds and $b$-maps, which are respectively the objects and morphisms of the $b$-category, revisit the notions of $b$-tangent and $b$-cotangent bundles. We then define differential $b$-forms and construct a $b$-de Rham cohomology theory.

\subsection{The $b$-tangent and $b$-cotangent bundles}

We begin with some definitions.

\begin{definition}
A \textbf{$b$-manifold} is a pair $(M,Z)$ of an oriented manifold $M$ and an oriented hypersurface $Z\subset M$. A \textbf{$b$-map} is a map $f:(M_1,Z_1)\rightarrow(M_2,Z_2)$ transverse to $Z_2$ and such that $f^{-1}(Z_2)=Z_1$. The \textbf{$b$-category} is the category whose objects are $b$-manifolds and morphisms are $b$-maps.
\end{definition}

\begin{definition}
A \textbf{$b$-vector field} on $M$ is a vector field which is tangent to $Z$ at every point $p\in Z$.
\end{definition}

These vector fields form a Lie subalgebra of the algebra of all vector fields on $M$. Moreover, they also form a projective module over the ring $C^\infty(M)$, and hence are sections of a vector bundle on $M$. We call this vector bundle the \textbf{$b$-tangent bundle} and denote it $^b TM$.

If $v$ is a $b$-vector field, then the restriction $v|_Z$ is everywhere tangent to $Z$, and hence defines a vector field $v_Z$ on $Z$. Thus we have a map $\Gamma(^b TM|_Z)\longrightarrow\Gamma(TZ)$ and since this map is a morphism of $C^\infty(Z)$-modules, it is induced by a vector bundle morphism
\begin{equation}\label{tangentbundlemorphism}
^b TM|_Z\longrightarrow TZ.
\end{equation}

\begin{proposition}
The kernel of the map (\ref{tangentbundlemorphism}) is a line bundle $\LL_Z$ with a canonical nonvanishing section.
\end{proposition}

\begin{proof}
Since $Z$ is oriented, there exists a defining function for $Z$, i.e.\ a $b$-map $f:(M,Z)\rightarrow(\RR,0)$. Let $v$ be a vector field on $M$ with $df_p(v_p)=1$ for all $p\in Z$. Then $w=f v|_Z$ is a nonvanishing section of $\LL_Z$, and it is easily checked that its definition doesn't depend on the choice of the function $f$ or the vector field $v$.
\end{proof}

We call this nonvanishing section $w$ of $\LL_Z$ the \textbf{normal $b$-vector field} of the $b$-manifold $(M,Z)$.

Note that at points $p\in M\setminus Z$, the $b$-tangent space coincides with the usual tangent space $^b T_p M=T_p M$, whereas at points $p\in Z$, there is a surjective map
$$^b T_p M\rightarrow T_p Z$$
with kernel spanned by $w_p$.

We define the \textbf{$b$-cotangent bundle} of $M$ to be the vector bundle $^b T^*M$ dual to $^b TM$. Then, at points $p\in M\setminus Z$, the $b$-cotangent space coincides with the usual cotangent space: $^b T^*_p M=T^*_p M$. At points $p\in Z$, there is an embedding
$$T_p^*Z\rightarrow  \,^b T^*_p M$$
whose image is $\left\{l\in ^b T^*_p M| l(w_p)=0\right\}$.

Given a defining function $f$ for $Z$, let $\mu\in\Omega^1(M\setminus Z)$ be the one-form $\frac{df}{f}$. If $v$ is a $b$-vector field then the pairing $\left\langle v,\mu\right\rangle\in C^\infty(M\setminus Z)$ extends smoothly over $Z$ and hence $\mu$ itself extends smoothly over $Z$ as a section of $^b T^*M$. Moreover, $\mu_p(w_p)=1$ for $p\in Z$, so given $f$ we get a splitting
\begin{equation}\label{eq:splitting}
^b T^*_p M=T_p^* Z+\text{span}\left\{\mu_p\right\}.
\end{equation}
For ease of notation, we will write $\mu_p=\frac{df_p}{f}$, even though the expression on the right hand side is not well-defined for $p\in Z$.

\subsection{The $b$-de Rham complex}\label{bDeRham}

For each $k$, let $^b\Omega^k(M)$ denote the space of \textbf{$b$-de Rham $k$-forms}, i.e., sections of the vector bundle $\Lambda^k(^b T^*M)$. The usual space of de Rham $k$-forms sits inside this space in a somewhat nontrivial way: given $\mu\in\Omega^k(M)$, we interpret it as a section of $\Lambda^k(^bT^*M)$ by the convention
$$\mu_p\in\Lambda^k(T_p^*M)=\Lambda^k(^b T_p^*M)  \text{ at } p\in M\setminus Z$$
$$\mu_p=(i^*\mu)_p\in\Lambda^k(T_p^*Z)\subset\Lambda^k(^b T_p^*M) \text{ at } p\in Z$$
where $i:Z\hookrightarrow M$ is the inclusion map.

With these conventions it is easy to see that, having fixed a defining function $f$, every $b$-de Rham $k$-form can be written as
\begin{equation}\label{eq:bDeRham}
\omega=\alpha\wedge\frac{df}{f}+\beta, \text{ with } \alpha\in\Omega^{k-1}(M) \text{ and } \beta\in\Omega^k(M).
\end{equation}
Moreover, while $\alpha$ and $\beta$ are not unique, we claim:

\begin{proposition}\label{prop:unique}
At $p\in Z$, $\alpha_p$ and $\beta_p$ are unique.
\end{proposition}
\begin{proof}
This follows from the direct sum decomposition (\ref{eq:splitting}) and the fact that at $p\in Z$, $\alpha_p$ and $\beta_p$ have to be interpreted as elements of $\Lambda^{k-1}(T_p^*Z)$ and $\Lambda^k(T_p^*Z)$.
\end{proof}

The decomposition (\ref{eq:bDeRham}) enables us to extend the exterior $d$ operator to $^b\Omega(M)$ by setting
$$d\omega=d\alpha\wedge\frac{df}{f}+d\beta.$$
The right hand side is well defined and agrees with the usual exterior $d$ operator on $M\setminus Z$ and also extends smoothly over $M$ as a section of $\Lambda^{k+1}(^b T^*M)$. Note that $d^2=0$, which allows us to form a complex of $b$-forms, the $b$-de Rham complex:

$$0\rightarrow\,^b\Omega^0(M)\stackrel{d}{\longrightarrow}\,^b\Omega^1(M)\stackrel{d}{\longrightarrow}\,^b\Omega^2(M)\stackrel{d}{\longrightarrow}\ldots\rightarrow0$$

Also note that if $\omega=\alpha\wedge\frac{df}{f}+\beta\in\,^b\Omega^k(M)$ is closed, it is not necessarily true that $\alpha$ and $\beta$ are closed, take for example $\beta=0$ and $\alpha=\mu\wedge df$ with $\mu$ chosen so that $\alpha$ is not closed. {For example, for $k=2$ and $f$ the coordinate $x$, taking $\mu$ to be a different coordinate $y$ yields $\alpha=y \, dx$, which is not a closed form.}

Following the usual notation in the context of differential geometry, $\omega|_Z$ denotes a section $Z\to\, ^b T^*M|_Z$ (in contrast with $i^*\omega$, which would be a section $Z\to T^*Z$). In general, $b$-de Rham forms ``explode'' at $Z$. Those that vanish at $Z$ are in fact honest de Rham forms:

\begin{proposition}\label{honest}
If $\omega\in\,^b\Omega^k(M)$ is such that $\omega|_Z=0$, then $\omega\in\Omega^k(M)$.
\end{proposition}

\begin{proof}
Given $\omega=\alpha\wedge\frac{df}{f}+\beta\in\,^b\Omega^k(M)$, the condition $\omega|_Z=0$ as a section of $^b T^*(M)|_Z$ implies that $\alpha|_Z=0$ and $\beta|_Z=0$.
If  $\alpha|_Z=0$ then $\alpha_1:=\frac{\alpha}{f}$ is in $^b\Omega^{k-1}(M)$. Thus $\omega=\alpha_1\wedge df+\beta$ is in $\Omega^k(M)$.
\end{proof}

\begin{remark}\label{integrate}
Even though $b$-forms ``explode'' at $Z$, it is possible to integrate compactly supported $b$-forms of top degree over $M$: for $\omega\in\,^b\Omega^d(M)$ compactly supported we define the integral of $\omega$ over $M$ to be
$$\int_M\omega=\lim_{\varepsilon\to 0}\int_{|f|>\varepsilon}\omega,$$
where $f$ is a defining fuction for $Z$. This limit exists and is independent of the choice of $f$, the proof is similar to that of Theorem 2 in \cite{Radko}.
\end{remark}

\section{$b$-Symplectic Manifolds and $b$-Poisson Manifolds}\label{section:twopointsofview}
\subsection{$b$-symplectic manifolds}
In this section we introduce the notion of symplectic for the $b$-category, and observe that $b$-symplectic manifolds are also Poisson manifolds. A simple example is the $b$-manifold $(M,Z)$ where $M=\RR^{2n}$ with coordinates $x_1,y_1,\ldots,x_n,y_n$ and $Z$ is the hyperplane $y_1=0$. Consider the closed $b$-form
$$\omega=d x_1\wedge \frac{d y_1}{y_1}+\sum_{i=2}^n d x_i\wedge d y_i.$$
This form is non degenerate in the sense that $\omega^n$ is a well defined, nonvanishing $b$-form. That this example is the local prototype of all $b$-symplectic manifolds is the content of the $b$-Darboux theorem (Theorem \ref{theorem:Darboux2}).

Dualizing $\omega$ we obtain the bivector field
$$\Pi=y_1\frac{\partial}{\partial x_1}\wedge \frac{\partial}{\partial y_1}+\sum_{i=2}^{n} \frac{\partial}{\partial x_i}\wedge\frac{\partial}{\partial y_i},$$
which is a Poisson structure on $M$ since $\left[\Pi,\Pi\right]=0$. The symplectic foliation for this Poisson structure contains two open leaves --- the upper and lower half spaces given by $y_1>0$ and $y_1<0$ ---, and the union of the remaining leaves is the hyperplane $Z$, where $\Pi^n$ vanishes --- these leaves are $(2n-2)$-dimensional planes through the different levels of $x_1$.

\begin{definition}
Let $(M,Z)$ be a $2n$-dimensional $b$-manifold and $\omega\in\,^b\Omega^2(M)$ a closed $b$-form. We say that $\omega$ is \textbf{$b$-symplectic} if $\omega_p$ is of maximal rank as an element of $\Lambda^2(\,^b T_p^* M)$ for all $p\in M$.
\end{definition}

\begin{example} Analogous to what happens in the symplectic case, the $b$-cotangent bundle of a $b$-manifold $(M,Z)$ is a $b$-symplectic manifold. If $y_i$ are local coordinates for the manifold $M$ on a neighborhood of a point in $Z$, with $Z$ defined locally by $y_1=0$, and $x_i$ are the fiber coordinates on $^bT^*M$, then the canonical one-form is given in these coordinates by
$$x_1 \frac{d y_1}{y_1}+\sum_{i=2}^n x_i d y_i,$$
and its exterior derivative
$$d x_1\wedge \frac{d y_1}{y_1}+\sum_{i=2}^n d x_i\wedge d y_i$$
is a $b$-symplectic form on $^bT^*M$.
\end{example}

As seen in (\ref{eq:bDeRham}), fixing a defining function $f$ for $Z$ we can decompose the symplectic form as
\begin{equation}\label{eq:bsymplecticform}
\omega=\alpha\wedge\frac{df}{f}+\beta, \text{ with } \alpha\in\Omega^1(M) \text{ and } \beta\in\Omega^2(M).
\end{equation}

\begin{proposition}\label{thetamu}
Let $\tilde{\alpha}=i^*\alpha$ and $\tilde{\beta}=i^*\beta$, where $i:Z\hookrightarrow M$ is the inclusion then
the forms $\tilde{\alpha}$ and $\tilde{\beta}$ are closed. Furthermore:
 \begin{enumerate}[(a)]
\item The form $\tilde{\alpha}$ is nowhere vanishing and intrinsically defined in the sense that it does not depend on the splitting (\ref{eq:bsymplecticform}). In particular, the codimension-one foliation of $Z$ defined by $\tilde{\alpha}$ is intrinsically defined.

\item For each leaf $L\stackrel{i_L}\hookrightarrow Z$ of this foliation, the form $i^*_L\tilde{\beta}$ is intrinsically defined, and is a symplectic form on $L$.
\item In (\ref{eq:bsymplecticform}) we can assume without loss of generality that:

\begin{itemize}
\item $\alpha$ and $\beta$ are closed;
\item $\alpha\wedge\beta^{n-1}\wedge df$ is nowhere vanishing;
\item and in particular $i^*(\alpha\wedge\beta^{n-1})$ is nowhere vanishing.

\end{itemize}
\end{enumerate}
\end{proposition}

\begin{proof}

From (\ref{eq:bsymplecticform}) we get $d\alpha\wedge\frac{df}{f}+d\beta=0$ and by Proposition \ref{prop:unique} the forms $d\alpha_p$ and $d\beta_p$ are zero as elements of $\Lambda^2(T_p^* Z)$ and $\Lambda^3(T_p^* Z)$ for all $p\in Z$, i.e., $d i^*\alpha$ and $d i^*\beta$ are zero.

\begin{enumerate}[(a)]
\item As seen in (\ref{eq:splitting}), the cotangent space $^b T_p^* M$ for $p\in Z$ is the direct sum $T_p^*Z+\text{span}\left\{\frac{df_p}{f}\right\}$. If $\tilde{\alpha}_p=0$ we would have $\omega_p=\tilde{\beta}_p\in\Lambda^2(T_p^* Z)$ and consequently $\text{rank}(\omega_p)<2n$, so we conclude that $\tilde\alpha$ must be nonvanishing. Replacing $f$ by another defining function $g$, we have $f=g h$, with $h$ nonvanishing on $Z$. Near $Z$ we have $\frac{df}{f}=\frac{dg}{g}+d(\log \left|h\right|)$ and (\ref{eq:bsymplecticform}) becomes
$$\omega=\alpha\wedge\frac{dg}{g}+\beta_1,$$
where $\beta_1=\beta+\alpha\wedge d(\log \left|h\right|)$. Then, for $\tilde{\beta}_1=\tilde{\beta}+\tilde{\alpha}\wedge d(\log \left|h\right|)|_Z$, \begin{equation}\label{eq:omegathetamu}
\omega|_Z=\tilde{\alpha}\wedge\frac{dg}{g}|_Z+\tilde{\beta}_1.
\end{equation}

\item For a leaf $L$ of the foliation defined by $\tilde{\alpha}$, by (\ref{eq:omegathetamu}) we have $i^*_L\tilde{\beta}_1=i^*_L\tilde{\beta}$. If $i^*_L\tilde{\beta}$ were of rank smaller than $(2n-2)$ at some $p\in Z$, then $\omega_p=\tilde{\alpha}_p\wedge\frac{df_p}{f}$ would be of rank smaller than $n$ as an element of $^b \Lambda^2(T_p^* M)$.

\item Let $\mathcal{U}\cong Z\times(-1,1)$ be a tubular neighborhood of $Z$ such that the defining function $f$ is given locally by $\mathcal{U}\approx Z\times (-1,1)\longrightarrow (-1,1)$. Let $\pi$ be the projection $\mathcal{U}\cong Z\times (-1,1)\longrightarrow Z$. Because $i^*\alpha$ is closed we can write
$$\alpha=\pi^*i^*\alpha+h df$$
for some $h\in \mathcal C^{\infty}(\mathcal{U})$. By replacing $\alpha$ by $\pi^*i^*\alpha$ in (\ref{eq:bsymplecticform}) we can assume that $\alpha$ is closed, and from (\ref{eq:bsymplecticform}) it now follows that $\beta$ is closed. Moreover, since $\omega^n$ is nowhere vanishing as a section of $\Lambda^2(^b T^* M)$ the form $\alpha\wedge\beta^{n-1}\wedge df$ is nowhere vanishing.

\end{enumerate}
\end{proof}

\begin{remark} By Proposition \ref{prop:unique}, the two-form $\tilde{\beta}$ is a symplectic invariant of the pair $(\omega,f)$. Moreover, the proof of Proposition \ref{thetamu} (a) describes how $\tilde{\beta}$ depends on the choice of the defining function. describes how $\tilde{\beta}$ depends on the choice of the defining function.
\end{remark}

\begin{remark}
It is easy to see that $\tilde\alpha$ and $\tilde\beta$  are intrinsically defined since  $\alpha$ and $ \beta$ are intrinsically defined modulo summands of the form $hdf$, $h\in\mathcal C^\infty(U)$ and $\mu\wedge df$, $\mu\in \Omega^1(U)$.

\end{remark}

\begin{remark}\label{remark}
Since $\tilde\beta$ is closed and $\tilde{\beta}^{n-1}$ is non-vanishing, the Darboux theorem for ${\tilde\beta}$  tells us that locally around each point $p\in Z$ there exist coordinates $x_1, x_2, y_2, \dots, x_n, y_n$ such that $\tilde{\beta}=\sum_{i=2}^n d x_i\wedge dy_i$. Because $\tilde{\alpha}$ is locally exact and $\tilde{\alpha}\wedge \tilde{\beta}^{n-1}$ is nowhere vanishing, we can additionally assume that locally  $\alpha=d x_1$. Setting $y_1=f$ we have:
\begin{equation}\label{eq:darbouxz}
\omega_{\vert Z}=dx_1\wedge \frac{dy_1}{y_1}+\sum_{i=2}^n dx_i\wedge dy_i,
\end{equation}
and for the dual bivector field $\Pi\in \Gamma(\Lambda^2(^b T^* M))$:
$$\Pi_{\vert Z}=\frac{\partial}{\partial x_1}\wedge y_1\frac{\partial}{ \partial y_1} +\sum_{i=2}^n \frac{\partial}{\partial x_i}\wedge\frac{\partial}{ \partial y_i}$$
This shows how the Poisson bivector field $\Pi$ on $M$ induces a regular Poisson structure on $Z$ whose symplectic leaves are the level sets of $y_1$ with symplectic structure on each leaf given by $\sum_{i=2}^n dx_i\wedge dy_i$. Stated more intrinsically, $L\stackrel{i_L}\hookrightarrow Z$ is a leaf of this symplectic foliation if and only if $i^*_L\tilde{\alpha}=0$, the symplectic structure on $L$ is given by $i^*_L\tilde{\beta}$.

A final comment regarding (\ref{eq:darbouxz}): we will show below that this \lq\lq Darboux theorem\rq\rq for $\omega_{\vert Z}$ is a consequence of a Darboux theorem for $\omega$ itself (see section \ref{sec:darboux} , Theorem \ref{theorem:Darboux2}).

\end{remark}
\subsection{$b$-Poisson manifolds}

We now look at the Poisson counterparts of $b$-symplectic manifolds, and will prove in the next section that these two notions are equivalent.
This notion corresponds to non-degenerate Poisson structures on $b$-manifolds which for short we call $b$-Poisson structures. One could consider other Poisson structures on pairs $(M,Z)$ but the ones described below are the object of this paper.

\begin{definition}\label{definition:firstb}
Let $(M^{2n},\Pi)$ be an oriented Poisson manifold such that the map
$$p\in M\mapsto(\Pi(p))^n\in\Lambda^{2n}(TM)$$
is transverse to the zero section, then $Z=\{p\in M|(\Pi(p))^n=0\}$ is a hypersurface and we say that $\Pi$ is a \textbf{$b$-Poisson structure} on $(M,Z)$ and $(M,Z)$ is a \textbf{$b$-Poisson manifold}.
\end{definition}

In this subsection we give some examples (compact and otherwise) of $b$-Poisson structures. In section \ref{section:theextensionproblem} we will find a
procedure to construct examples of $b$-Poisson manifolds from regular Poisson manifolds with certain vanishing obstruction classes.

\begin{example}
The Lie algebra of the affine group of dimension 2 is a model for noncommutative Lie algebras in dimension 2, its algebra structure is given by $[e_1,e_2]=e_2$, where $e_1,e_2$ form a basis. We can naturally write this (bilinear) Lie algebra structure as the Poisson structure
$$\Pi=y\frac{\partial}{\partial x }\wedge \frac{\partial}{\partial y },$$
which is dual to the $b$-symplectic form $\omega=dx\wedge \frac{dy}{y}$. In this example the exceptional hypersurface $Z$ is the $x$-axis, it is the union of symplectic leaves of dimension 0 (points on the line), and the upper and lower half-planes are symplectic open leaves of dimension 2.
\end{example}

\begin{example}
On the sphere $\SSS^2$ with the usual coordinates $(h,\theta)$, the Poisson structure $\Pi=h\frac{\partial}{\partial h}\wedge\frac{\partial}{\partial \theta}$ vanishes transversally along the equator $Z=\{h=0\}$ and hence is a $b$-Poisson structure on $(\SSS^2,Z)$. {More generally, take any orientable surface $M$ and a curve $Z$ on it with defining function $f$. Let $\Omega$ a volume form on $M$, and $\Pi_\Omega$ be the bivector field dual to that form. Then $f\Pi_\Omega$ is a $b$-Poisson structure on $(M,Z)$.}

\end{example}

In dimension 2, these structures were studied and classified by Radko in  \cite{Radko} under the name of \emph{topologically stable Poisson structures}. In this case, $Z$ is a union of smooth curves and each point in these curves is a symplectic leaf of $Z$. In \cite{Radko}, Radko proves that the following ingredients give a complete classification of $b$-Poisson manifolds of dimension 2:

\begin{itemize}
\item The set of curves $\gamma_1,\dots,\gamma_n$ along which the Poisson structure vanishes;

\item The periods along the curves $\gamma_1,\dots,\gamma_n$ of a modular vector field\footnote{For a precise definition of modular vector field see section \ref{subsection:modularvectorfield}.} on $M$ associated to the volume form $\omega_{\Pi}$, the two-form dual to $\Pi$, on the complement of Z;

\item The regularized Liouville volume of $(M,\Pi)$, which is a correction along $Z$ of the natural volume associated to the Poisson structure, necessary because the original volume form $\omega_\Pi$ associated to the Poisson structure $\Pi$ blows up at $Z$. This regularized Liouville volume is  the integral $\int_M\omega_\Pi$ as defined in Remark \ref{integrate}.
\end{itemize}

The following classification holds:

\begin{theorem}[Radko] The set of curves, modular periods and regularized Liouville volume completely determines, up to Poisson diffeomorphisms, the $b$-Poisson structure on a compact surface $M$.
\end{theorem}

In the paper \cite{Radko}, the author then uses this classification to explicitly compute the Poisson cohomology of the manifold in terms of the modular vector field.

\begin{example}
So far our examples have been of dimension 2, but product structures allow us to get to higher dimensions: let $(R,\pi_R)$  be a Radko compact surface and $(S,\pi_S)$ be a compact symplectic surface, then $(R\times S,\pi_R+\pi_S)$ is a $b$-Poisson manifold of dimension 4. Furthermore, perturbations of product structures can be used to obtain non-product ones. For instance, take $\SSS^2$ with $b$-Poisson structure $\Pi_1=h\frac{\partial}{\partial h}\wedge\frac{\partial}{\partial \theta}$ and the symplectic torus $\mathbb T^2$ with dual Poisson structure $\Pi_2=\frac{\partial}{\partial\theta_1} \wedge\frac{\partial}{\partial \theta_2}$. Then
$$\hat\Pi=h \frac{\partial}{\partial h}\wedge(\frac{\partial}{\partial \theta}+\frac{\partial}{\partial\theta_1})+ \Pi_2.$$
is a $b$-Poisson structure on  $\SSS^2\times \mathbb T^2$, but not the product one described above.

We will see in Section \ref{section:normalforms} that using Moser path methods in the $b$-context, we can control perturbations that produce equivalent Poisson structures.
\end{example}

\begin{example}\label{basicexample}
Let $(N^{2n+1},\pi)$ be a regular corank-1 Poisson manifold, $X$ be a Poisson vector field and $f:\SSS^1\to\RR$ a smooth function. The bivector field
$$\Pi=f(\theta)\frac{\partial}{\partial\theta}\wedge X+\pi$$
is a $b$-Poisson structure on $\SSS^1\times N$ if the function $f$ vanishes linearly and the vector field $X$ is transverse to the symplectic leaves of $N$ (this condition is necessary to guarantee transversality). If that is so, the exceptional hypersurface consists of the union of as many copies of $N$ as zeros of $f$.

In this example $N^{2n+1}$ is the exceptional hypersurface of the $b$-Poisson manifold and has an induced Poisson structure which is regular of corank one. It is a general fact that the exceptional hypersurface of a $b$-Poisson manifold naturally inherits a corank-one Poisson structure.
This example provides the semilocal model for a $b$-Poisson structure in a neighborhood of the exceptional hypersurface $Z$.

\end{example}

\subsection{$b$-Poisson equals $b$-symplectic}\label{subsection:poisson b symplectic}

In this section, we will show the following:

\begin{proposition}\label{prop:poissonbsymplectic}
A two-form $\omega$ on a $b$-manifold $(M,Z)$ is $b$-symplectic if and only if its dual bivector field $\Pi$ is a $b$-Poisson structure.
\end{proposition}

{Note that because $\omega$ is of maximal rank in $\Lambda^2(^bT^*M)$ and $\Pi$ is of maximal rank in $\Lambda^2(^bT^*M)$, it makes sense to say that they are dual to each other. Similarly, this it makes sense to say that a volume form (of maximal rank in $\Lambda^{2n}(T^*M)$) has a dual $2n$-vector field (of maximal rank in $\Lambda^{2n}(TM)$).}

We begin by recalling Weinstein's splitting theorem, which we will then apply to the particular case of $b$-Poisson manifolds:

\begin{theorem}[Weinstein]\label{thm:splitting}
Let $(M^m,\Pi)$ be a Poisson manifold of rank $2k$ at a point $p\in M$. Then there exists a neighborhood and a local coordinate system $(x_1,y_1,\ldots, x_{k},y_{k}, z_1,\ldots, z_{m-2k})$ centered at $p$ for which the Poisson structure can be written as
\begin{equation}
\Pi = \sum_{i=1}^k \frac{\partial}{\partial x_i}\wedge
\frac{\partial}{\partial y_i} + \sum_{i,j=1}^{m-2k}
f_{ij}(z)\frac{\partial}{\partial z_i}\wedge
\frac{\partial}{\partial z_j} ,
\end{equation}
where $f_{ij}$ are functions which depend only on the variables $(z_1,\ldots,z_{m-2k})$ and which vanish at the origin.
\end{theorem}

Let $\Pi$ be a $b$-Poisson structure on $(M,Z)$, and $\Omega$ and $\Xi$ a volume form on $M$ and its dual $2n$-vector field, respectively. Then, $\Pi^n= f \Xi$ for some $f:M\rightarrow\RR$ which vanishes on $Z$. Since the $2n$-vector field $\Pi^n$ doesn't vanish identically, the generic rank of the Poisson structure is $2n$, and it is less than $2n$ on $Z$. This implies that the two-form $\omega_\Pi$ dual to $\Pi$ is a smooth symplectic form on $M\setminus Z$. Because $\Pi^n$ intersects the zero section of $\bigwedge^n(TM)$ transversally, $0$ is a regular value of $f$ and so $Z=f^{-1}(0)$ must be a codimension-one submanifold of $M$, a union of hypersurfaces. Furthermore, we can assume that in a neighborhood of a point in $Z$, the function $f$ is simply the coordinate function $z_1$, with $z_1=0$ locally defining the hypersurface. When restricted to $Z$, the Poisson structure defines a symplectic foliation of codimension one. To prove this, use Theorem \ref{thm:splitting} and observe that $\Pi^n$ vanishing transversally at $Z$ implies that the transverse Poisson manifolds at points of $Z$ must be of dimension two, so $Z$ is the union of symplectic leaves of corank 2 in $M$. This defines a codimension-one foliation of $Z$ by symplectic leaves.

Summarizing, a $b$-Poisson manifold is a Poisson manifold $(M^{2n},\Pi)$ for which the $2n$-vector field $\Pi^n$ vanishes linearly along a disjoint union of smooth hypersurfaces $Z$ and such that the Poisson structure $\Pi$ defines a symplectic structure $\omega_\Pi$ on $M\setminus Z$ and when restricted to $Z$ gives a symplectic foliation of codimension one in $Z$. In particular, the rank maximality of a $b$-symplectic form implies that its dual bivector field must be a $b$-Poisson structure; this proves one of the directions of Proposition \ref{prop:poissonbsymplectic}.

Consider now the particular case of a 2-dimensional $b$-Poisson manifold $(M,Z)$. The Poisson bivector field $\Pi$ vanishes linearly at $Z$, and the dual two-form will be given locally by $\omega_{\Pi}=\frac{1}{z_1 f(z_1,z_2)}\,dz_1\wedge dz_2$, with $f(z_1,z_2)$ nonvanishing on $Z$, which is given locally by $z_1=0$. The diffeomorphism $\phi$ given by the change of coordinates $z=z_1$ and $t=\int f(z_1,z_2)dz_2$ satisfies $\phi^*(\omega_{\Pi})=\frac{1}{z}\,dz\wedge dt$ (here we give the explicit diffeomorphism, but the existence of such a diffeomorphism derives simply from the fact that $\Pi^n$ intersects the zero section
$\bigwedge^n(TM)$ transversally and uses the regular value theorem).

Combining this result with Theorem \ref{thm:splitting}, we obtain a local normal form result resembling the Darboux theorem, which will appear again as Theorem \ref{theorem:Darboux2} and that we will reprove then using Moser path methods.

\begin{proposition}\label{prop:localb}
Let $(M,Z)$ be a $b$-Poisson manifold, with Poisson bivector field $\Pi$ and dual two-form $\omega_\Pi$. Then, on a neighborhood of a point $p\in Z$, there exist coordinates $(x_1,y_1,\dots x_{n-1},y_{n-1}, z, t)$  centered at $p$ such that
$$\omega_{\Pi}=\sum_{i=1}^{n-1} dx_i\wedge dy_i+\frac{1}{z}\,dz\wedge dt.$$
\end{proposition}

\begin{remark}\label{remark:darbouxpoisson}
In other words, we can find a splitting such that
$$\omega_{\Pi}=\omega_L+(\Pi^T)^\sharp$$
where $\omega_L$ is the symplectic form on the symplectic leaf through the point $p\in Z$ and $(\Pi^T)^\sharp$ is the dual to a $b$-Poisson structure on a 2-dimensional manifold. In particular, we have obtained a linearization result for bivector fields associated to $b$-manifolds.
\end{remark}

Because being symplectic is a local property, Proposition \ref{prop:localb} implies that a $b$-Poisson manifold is $b$-symplectic, the other direction of Proposition \ref{prop:poissonbsymplectic}. From now on, we will refer to these manifolds as $b$-symplectic manifolds.

\subsection{Modular vector fields of $b$-symplectic manifolds}\label{subsection:modularvectorfield}

A modular vector field on a Poisson manifold measures how far Hamiltonian vector fields are from preserving a given volume form. A simple example is that of a symplectic manifold endowed with the volume form that is the top power of its symplectic form: the modular vector field will be zero because this volume form is invariant under the flow of any Hamiltonian vector field. In this section we study modular vector fields of $b$-symplectic manifolds.

We follow Weinstein \cite{Weinstein2} for the description of modular vector fields of Poisson manifold; a complete presentation of these can also be found in \cite{kosman}. Some results about modular vector fields for regular corank one Poisson manifolds, as is the case of the exceptional hypersurface $Z$ of a $b$-symplectic manifold $(M,Z)$, can be found in \cite{guimipi}.

\begin{definition}
Let $(M,\Pi)$ be a Poisson manifold and $\Omega$ a volume form on it, and denote by $u_f$ the Hamiltonian vector field associated to a smooth function $f$ on $M$. The \textbf{modular vector field} $X_{\Pi}^{\Omega}$ (or simply $X^{\Omega}$ if the Poisson structure is fixed and hence implicit) is the derivation given by the mapping
$$f\mapsto \frac{\cL_{u_f}\Omega}{\Omega}.$$
\end{definition}

Let $(M,Z)$ be a $b$-symplectic manifold and consider the local coordinates given by Proposition \ref{prop:localb} in a neighborhood of a point $p\in Z$. The $b$-symplectic form $\omega$ can be written as
$$\omega=\sum_{i=1}^{n-1} dx_i\wedge dy_i+\frac{1}{z}\,dz\wedge dt,$$
and consider also the volume form
$$\Omega=dx_1\wedge dy_1\wedge\dots\wedge  dx_{n-1}\wedge dy_{n-1}\wedge dz\wedge dt.$$

Working in these local coordinates we see that the modular vector field associated to the volume form $\Omega$ and the Poisson structure dual to the $b$-symplectic form is given by
$$X^{\Omega}=\frac{\partial}{\partial t}.$$

As a consequence we have:

\begin{proposition}
The modular vector field of a $b$-symplectic manifold $(M,Z)$ is tangent to $Z$ and transverse to the symplectic leaves inside $Z$, independently of the volume form considered on $M$.
\end{proposition}

\begin{proof}
At each point $p\in Z$, working in the local coordinates mentioned above, the modular vector field with respect to the volume form $\Omega$ above is
$$X^{\Omega}=\frac{\partial}{\partial t}.$$
This vector field is tangent to $Z$, which is given locally by $z=0$, and transverse to the symplectic foliation inside $Z$, because the leaves of that foliation are locally just the different levels of the coordinate function $t$.

If we consider another volume form $\Omega'=H \Omega$, where $H\in C^\infty(M)$ is nonvaninshing, the modular vector field becomes
$$X^{\Omega'}=\frac{\partial}{\partial t}+ u_{log(\left|H\right|)},$$
it differs from the previous one by a hamiltonian vector field\footnote{The \textbf{modular class} of a Poisson manifold is the class of a modular vector field in the first Poisson cohomology group, it depends on the Poisson structure but not on the volume form: a change in the volume form changes the modular vector field by a hamiltonian vector field.}. Hamiltonian vector fields are tangent to the symplectic leaves of $M$, and in particular, to all the $(2n-2)$-dimensional leaves whose union is $Z$ and hence to $Z$ itself. Therefore, the new modular vector field $X^{\Omega'}$ will still be tangent to $Z$ and transverse to the symplectic leaves in it.
\end{proof}

\begin{remark}\label{rem:remark}
Using these local coordinates we also see that $\tilde\alpha(v_\text{mod}|_Z)=1$, independently of choice of modular vector field $v_\text{mod}$.
\end{remark}

\section{Cohomology theories for $b$-manifolds}\label{section:cohomology}

In this section we explore some cohomology theories for $b$-manifolds, and the relationships between them. For a $b$-manifold $(M,Z)$, we can talk about the usual cohomology theories for the underlying manifold $M$, such as de Rham cohomology and Poisson cohomology, which correspond respectively to de Rham forms and to multivector fields, but we can also use the notions of $b$-forms and $b$-multivector fields, and study the corresponding cohomology theories.

\subsection{De Rham cohomology and $b$-cohomology}

We begin by proving a Mazzeo-Melrose theorem for $b$-manifolds (see \cite[\S 2.16]{Melrose} for the original version) and then as a direct application obtain some results about low degree $b$-cohomology for $b$-symplectic manifolds.

Let $(M,Z)$ be a $b$-manifold with  $Z\stackrel{i}\hookrightarrow M$ compact.

\begin{theorem}\label{thm:mazzeomelrose}\textbf{[$b$-Mazzeo-Melrose theorem]}
The $b$-cohomology groups of $M$ are computable by
$$^b H^*(M)\cong H^*(M)\oplus H^{*-1}(Z).$$
\end{theorem}

\begin{proof}
Let $f:(M,Z)\rightarrow(\RR,0)$ be a defining function for $Z$. Then, every $\omega\in \,^b\Omega^k(M)$ can be written as $\omega=\alpha\wedge\frac{df}{f}+\beta$, with $\alpha,\beta\in\Omega^*(M)$. Moreover, in $b$-de Rham theory, $\tilde{\alpha}:=i^*\alpha$, where $i:Z\hookrightarrow M$ is the inclusion, is intrinsically defined independent of the choice of $f$, so we have a canonical short exact sequence of de Rham complexes
$$0\rightarrow\Omega^k(M)\rightarrow \,^b\Omega^k(M)\rightarrow\Omega^{k-1}(Z)\rightarrow0$$
which induces a long exact sequence in cohomology
$$\ldots \rightarrow H^k(M)\stackrel{i}{\rightarrow} \,^b H^k(M)\stackrel{j}{\rightarrow}H^{k-1}(Z)\stackrel{\delta}{\rightarrow}H^{k-1}(M)\rightarrow\ldots$$

This sequence splits into short exact sequences because the map $j$ is surjective: let $\cU\cong Z\times(-\varepsilon,\varepsilon)$ be a collar neighborhood of $Z$ in $M$, change $f$ so that $f\equiv1$ on the complement of $Z\times(-\frac{\varepsilon}{2},\frac{\varepsilon}{2})$, and let $p:Z\times(-\frac{\varepsilon}{2},\frac{\varepsilon}{2})\rightarrow Z$ be the projection. Then, for any closed $\tilde{\alpha}\in\Omega^{k-1}(Z)$, the form $\omega=p^*\tilde{\alpha}\wedge\frac{df}{f}\in \,^b\Omega^k(M)$, and $j[\omega]=[\tilde{\alpha}]$.
\end{proof}

Observe that if $M$ is compact then the theorem above says that for cohomology of top dimension we have
$$^bH^d(M)=H^d(M)\oplus(\oplus_i H^{d-1}(Z_i))$$
where the $Z_1,Z_2,\ldots,Z_r$ are the connected components of $Z$, and hence $\dim(\, ^bH^d(M))=r+1$.

The symplectic form on a compact symplectic manifold defines a non-vanishing second cohomology class on the manifold, yielding a non-trivial second cohomology group. For $b$-symplectic manifolds, the analogue involves the second $b$-cohomology group:

\begin{proposition}
For a compact $b$-symplectic manifold $(M,Z)$ we have $H^1(Z)\neq\{0\}$ and consequently $^b H^2(M)\neq\{0\}$.
\end{proposition}

\begin{proof}
The submanifold $Z$, being the vanishing set of $\Pi^n$ (the top power of the bivector field dual to the $b$-symplectic form), is closed and since $M$ is compact, $Z$ is compact as well. Let $\alpha$ be a one-form $Z$ that defines the corank-1 regular foliation induced by $\Pi$ on $Z$: the form $\alpha$ is nowhere vanishing and $i^*_L\alpha=0$ for all leaves $L\stackrel{i_L}\hookrightarrow Z$. If we had $H^1(Z)=0$ then $\alpha$ would be exact: $\alpha= dg$ for some function $g\in C^\infty(Z)$. By compactness of $Z$, the function $g$ has maximum and minimum points, at which $\alpha=dg$ would then necessarily vanish.

Then, by Theorem \ref{thm:mazzeomelrose} we have
$$^b H^2(M)\cong H^2(M)\oplus H^{1}(Z)\neq \{0\}.$$
\end{proof}

\begin{proposition}
For a compact $b$-symplectic manifold $(M,Z)$ we have $H^2(Z)\neq \{0\}$ and consequently $^b H^3(M)\neq\{0\}$.
\end{proposition}

\begin{proof}
Suppose that $H^2(Z)=0$. Then $d\tilde\beta=0$ on $Z$ implies that $\tilde\beta=d\mu$ for some one-form $\mu$ defined on $Z$. Since $\tilde\beta^n\wedge\tilde\alpha$ is a volume form on $Z$ (see \cite{guimipi}), we have
$$0 \neq \text{vol}(Z) = \int_Z \tilde\beta^n\wedge\tilde\alpha = \int_Z (d\mu)^n\wedge\tilde\alpha = \int_{\partial Z}\mu\wedge(d\mu)^{n-1}\wedge\tilde\alpha = 0.$$
Thus $H^2(Z)$ must be non-trivial and Theorem \ref{thm:mazzeomelrose} gives us $^b H^3(M)\neq\{0\}$.
\end{proof}

\subsection{Poisson cohomology and $b$-Poisson cohomology}

In this section we study the relationship between $b$-cohomology and Poisson cohomology as a direct application rederive the computation of Poisson cohomology for two-dimensional $b$-manifolds obtained in \cite{Radko}.

We first prove that the cohomology of the Lichnerowicz complex associated to multivectorfields which are tangent to the exceptional hypersurface $Z$ is isomorphic to the $b$-cohomology. We then use a lemma of Marcut and Osorno \cite{marcutosorno2} to prove that this cohomology indeed does compute Poisson cohomology. A different approach for this computation would be to use a K�nneth and Mayer-Vietoris-type argument for the Poisson cohomology associated to a $b$-symplectic form. But this one which uses the key lemma in  \cite{marcutosorno2} seems to provide the shortest path.

For a general Poisson manifold $(M,\Pi)$, the Poisson structure $\Pi$ induces a differential operator $d_\Pi=[\Pi,\cdot]$ on the graded algebra of multivector fields on $M$ by extending the Lie bracket to multivectorfields. The cohomology of the complex of multivector fields $\Lambda^*(M)$
$$\ldots\longrightarrow\Lambda^{k-1}(M)\stackrel{d_\Pi}{\longrightarrow}\Lambda^{k}(M)\stackrel{d_\Pi}{\longrightarrow}\Lambda^{k+1}(M)\longrightarrow\ldots$$
is the Poisson cohomology $H_\Pi^*(M)$ of $M$ associated with the Poisson structure $\Pi$.

Let $^b\Lambda^k(M)$ denote the space of \textbf{$b$-multivector fields}, i.e., sections of the vector bundle $\Lambda^k(^b TM)$. Then the operator $d_\Pi=[\Pi,\cdot]$ is a differential on the subalgebra of $b$-multivector fields on $M$. The \textbf{$b$-Poisson cohomology} $^b H_\Pi^*(M)$ associated to the $b$-Poisson structure $\Pi$ on $M$ is the cohomology of the complex
$$\ldots\longrightarrow\,^b\Lambda^{k-1}(M)\stackrel{d_\Pi}{\longrightarrow}\,^b\Lambda^{k}(M)\stackrel{d_\Pi}{\longrightarrow}\,^b\Lambda^{k+1}(M)\longrightarrow\ldots$$

Explicit computations of Poisson cohomology are close to impossible in the general Poisson case. A simple example in which Poisson cohomology can be computed is the case of symplectic manifolds, for which the Poisson cohomology is isomorphic to de Rham cohomology. This is because non-degeneracy of the symplectic form allows us to define a bundle isomorphism between $T^*M$ and $TM$. Similarly, in the $b$-symplectic case, non-degeneracy of the $b$-symplectic form gives a bundle isomorphism between $^b T^*M$ and $^b TM$, which translates to an isomorphism between $b$-de Rham cohomology and $b$-Poisson cohomology.

\begin{theorem}\label{thm:liealgebroid}
Let $(M,Z)$ be a $b$-symplectic manifold, and $\Pi$ the corresponding $b$-Poisson structure. Then, the $b$-Poisson cohomology $^b H_\Pi^*(M)$ is isomorphic to the $b$-de Rham cohomology $^b H^*(M)$.
\end{theorem}

\begin{proof}
We define the operator $\natural: \,^bT^*M\rightarrow\, ^b TM$ such that for $\alpha, \beta\in\,^b\Omega^1(M)$ we have
$$<\alpha\wedge \beta,\Pi >=<\beta, \natural(\alpha)>.$$
By taking exterior powers of $\natural$ we obtain for each $k$ a homomorphim between $\Lambda^k(^b T^*M)$ and $\Lambda^k(^b TM)$, and hence also between $^b\Omega^k(M)$ and $^b\Lambda^k(M)$. Because the Poisson structure $\Pi$ is non-degenerate in this $b$-context, this homomorphism is an isomorphism, which we will also denote by $\natural$.

The classical formula for Lichnerowicz complexes
$$\natural(d\eta)=-[\Pi, \natural(\eta)]=-d_{\Pi}(\natural(\eta))$$
guarantees that the following diagram commutes

\xymatrix{
&...\ar[r]& \,^b\Lambda^{k-1}(M) \ar[r]^{d_{\pi}}& \,^b\Lambda^k(M) \ar[r]^{d_{\pi}} &\,^b \Lambda^{k+1}(M) \ar[r] &...
\\
&...\ar[r]& \,^b\Omega^{k-1}(M) \ar[u]^\natural \ar[r]^{d} & \,^b\Omega^k(M) \ar[u]^\natural \ar[r]^{d} & \,^b\Omega^{k+1}(M) \ar[u]^\natural \ar[r] &...,}
\noindent thus providing the desired isomorphism between the cohomologies of the two complexes.
\end{proof}

We are interested not only in the $b$-Poisson cohomology of a $b$-symplectic manifold $(M,Z)$, but also in its honest Poisson cohomology.  As it is observed in \cite{marcutosorno2} , the Lichnerowicz complex associated to these multivector fields (the ones tangent to $Z$) is a subcomplex of the complex computing Poisson cohomology of $M$. The following key lemma (which is proved in \cite{marcutosorno2}) shows that the inclusion $^b\Lambda^{k}(M)\subset \Lambda^{k}(M)$ induces an isomorphism in cohomology.

The key point to prove this lemma is to consider the normal form given by equation \ref{eq:normalform} in Section \ref{section:theextensionproblem}.

\begin{lemma}[Marcut-Osorno] \label{lem:marcut} The inclusion  $^b\Lambda^{k}(M)\subset \Lambda^{k}(M)$ induces an isomorphism
in cohomology and thus,

$$^b H_\Pi^*(M)\cong  H_\Pi^*(M)$$

\end{lemma}

Now combining Theorem \ref{thm:liealgebroid} with Lemma \ref{lem:marcut}, we obtain the following result, which is Proposition 1 in Section 5 of \cite{marcutosorno2}.

\begin{theorem}\label{thm:poissoncohomology}
Let $(M,Z)$ be a compact $b$-symplectic manifold, $\Pi$ the corresponding $b$-Poisson structure.
Then the Poisson cohomology groups of $M$ are computable by
$$H^k_{\Pi}(M)\cong \,^bH^k(M)$$
\end{theorem}

In the two-dimensional case, we can use  Theorem \ref{thm:poissoncohomology} to reprove a result of  Radko's \cite{Radko}.

\begin{corollary} [Radko]
Let $(M,Z)$ be a compact connected two-dimensional $b$-symplectic manifold, where $M$ is of genus $g$ and $Z$ a union of $n$ curves on $M$. Then the Poisson cohomology of $M$ is given by
\begin{align*}
H_{\Pi}^0(M)&=\mathbb R\\
H_{\Pi}^1(M)&=\mathbb R^{n+2g}\\
H_{\Pi}^2(M)&=\mathbb R^{n+1}.
\end{align*}
\end{corollary}

\begin{proof}
 By Theorems \ref{thm:poissoncohomology} and \ref{thm:mazzeomelrose} we have
$$H^k_{\Pi}(M)\cong \,^bH^k(M)\cong H^k(M)\oplus H^{k-1}(Z).$$
The result is then immediate from the fact that the nonzero cohomology groups of an oriented compact connected surface $M$ of genus $g$  are $H^0(M)=H^2(M)=\mathbb{R}$ and $H^1(M)=\mathbb{R}^{2g}$, and those of a union $Z$ of $n$ curves are $H^0(Z)=H^1(Z)=\mathbb{R}^n$.
\end{proof}

\section{Normal Forms}\label{section:normalforms}

With the necessary tools and notions now in place, we can prove $b$-analogues of standard symplectic geometry theorems.

\subsection{Relative Moser theorem for $b$-symplectic manifolds}

\begin{theorem}\label{Moser}
Let $\omega_0$ and $\omega_1$ be two $b$-symplectic forms on $(M,Z)$. If $\omega_0|_Z=\omega_1|_Z$, then there exist neighborhoods $\cU_0,\cU_1$ of $Z$ in $M$ and a diffeomorphism $\gamma:\cU_0\rightarrow\cU_1$ such that $\gamma|_Z=\text{id}_Z$ and $\gamma^*\omega_1=\omega_0$.
\end{theorem}

\noindent Recall that for $p\in Z$, one should interpret $\omega|_p$ as sitting in {$\Lambda^2(^b T_p^*M)$}.

\begin{proof}(using Moser trick)
Let $\omega_t=(1-t)\omega_0+t\omega_1$. We will prove that there exists a neighborhood $\cU$ of $Z$ in $M$ and an isotopy $\gamma_t:\cU\rightarrow M$, with $0\leq t\leq1$ such that $\gamma_t|_Z=\text{id}_Z$ and
\begin{equation}\label{gammat}\gamma_t^*\omega_t=\omega_0.\end{equation}

If such a $\gamma_t$ is to exist, by differentiating (\ref{gammat}) we will get
\begin{equation}\label{difgammat}\cL_{v_t}\omega_t=\omega_0-\omega_1,\end{equation}
where $v_t=\frac{d\gamma_t}{dt}\circ\gamma_t^{-1}$. Note that since $\gamma_t|_Z=\text{id}_Z$, we would have $v_t|_Z=0$, so $v_t$ would be a $b$-vector field vanishing on $Z$.

Because $(\omega_0-\omega_1)|_Z=0$, by Proposition \ref{honest} the $b$-form $(\omega_0-\omega_1)$ is also an honest de Rham form, and since it is closed, by the Poincar\'{e} lemma there exists a one-form $\mu\in\Omega^1(M)$ such that
$(\omega_0-\omega_1)=d(f\mu)$ on a neighborhood of $Z$, where $f:(M,Z)\to(\mathbb{R},0)$ is a defining function for $Z$. Then, (\ref{difgammat}) becomes
\begin{equation}\label{vt}\iota_{v_t}\omega_t=f\mu,\end{equation}
which is can be solved for $v_t$ in a small enough neighborhood $\cU$ of $Z$ where $\omega_t$ is $b$-symplectic {(note that such a neighborhood always exists, since $\omega_0|_Z=\omega_1|_Z\neq 0$)}. Moreover, since
the right hand side of (\ref{vt}) vanishes at points of $Z$, the $b$-vector field $v_t$ thus defined does too.

We can get a suitable $\gamma_t$ by integrating $v_t$, and the vector field vanishing on $Z$ implies that $\gamma_t|_Z=\text{id}_Z$ as desired. Now set
$\gamma:=\gamma_1$ and the open sets $\cU_0:=\cU$ and $\cU_1:=\gamma_1(\cU)$.
\end{proof}

An alternative statement of Theorem \ref{Moser} is the following:

\begin{theorem}\label{theorem:relativemoser2}
Let $\omega_0$ and $\omega_1$ be two $b$-symplectic forms on $(M,Z)$. If they induce on $Z$ the same restriction of the Poisson structure and their modular vector fields differ on $Z$ by a Hamiltonian vector field, then there exist neighborhoods $\cU_0,\cU_1$ of $Z$ in $M$ and a diffeomorphism $\gamma:\cU_0\rightarrow\cU_1$ such that $\gamma|_Z=\text{id}_Z$ and $\gamma^*\omega_1=\omega_0$.
\end{theorem}

\begin{proof}
If suffices to show that $\omega_0|_Z=\omega_1|_Z$, then apply Theorem \ref{Moser}.

Fix a defining function $f$ for $Z$ and write $\omega_j=\alpha_j\wedge\frac{df}{f}+\beta_j$, and also $\tilde{\alpha_j}=i^*\alpha_j$ and $\tilde{\beta_j}=i^*\beta_j$, with $j=0,1$ and $i:Z\hookrightarrow M$ the inclusion.
What we want to show is that $\tilde{\alpha_0}=\tilde{\alpha_1}$ and $\tilde{\beta_0}=\tilde{\beta_1}$. Note that $v_{\text{mod}\,j}$ being the modular vector field implies that $\tilde{\alpha_j}(v_{\text{mod}\,j}|_Z)=1$ and $\iota_{v_{\text{mod}\,j}|_Z}\tilde{\beta_j}=0$.

The pullback of $\tilde{\alpha}_j$ to each symplectic leaf in $Z$ vanishes (see Remark \ref{remark}). Therefore, in order to conclude that $\tilde{\alpha}_0=\tilde{\alpha}_1$ we just need to check that these forms agree when contracted with a vector field transversal to the leaves, for example $v_{\text{mod}\,0}|_Z$:  indeed, $\tilde{\alpha_0}(v_{\text{mod}\,0}|_Z)=1$ and
\begin{align*}
\tilde{\alpha_1}(v_{\text{mod}\,0}|_Z)&=\tilde{\alpha_1}(v_{\text{mod}\,1}|_Z + (v_{\text{mod}\,0}|_Z-v_{\text{mod}\,1}|_Z))\\
&= 1
\end{align*}
 because $(v_{\text{mod}\,0}|_Z-v_{\text{mod}\,1}|_Z)$ is a Hamiltonian vector field on $Z$ and therefore tangent to $Z$.

Because $\omega_0$ and $\omega_1$ induce the same restriction of Poisson structure on the hypersurface $Z$, we have $i^*_L\tilde{\beta_0}=i^*_L\tilde{\beta_1}$ for any symplectic leaf $L$, with $i^*_L: L\hookrightarrow Z$ the inclusion.
This, together $\iota_{v_{\text{mod}\,j}|_Z}\tilde{\beta_j}=0$, gives $\tilde{\beta_0}=\tilde{\beta_1}$.
\end{proof}

\begin{remark}
It is possible to give versions of these two relative Moser's theorems using $b$-cohomology (see for instance  \cite{scott}). Those statements require that the $b$-cohomology classes coincide but do not require that ${\omega_0}_{\vert Z}={\omega_1}_{\vert Z}$. However, the diffeomorphisms obtained do not necessarily restrict to the identity on $Z$.

\end{remark}
\subsection{Darboux theorem for $b$-symplectic manifolds}\label{sec:darboux}

As in the symplectic case, the relative Moser theorem can be used to prove a local canonical form result for $b$-symplectic forms, an analogue of the classical Darboux theorem.  A different proof of this result can also be found in Lemma 2.4 in \cite{NestandTsygan}.

\begin{theorem}\label{theorem:Darboux2}\textbf{[$b$-Darboux theorem]}
Let $\omega$ be a $b$-symplectic form on $(M,Z)$ and $p\in Z$. Then we can find a coordinate chart $(\cU,x_1,y_1,\ldots,x_n,y_n)$ centered at $p$ such that on $\cU$ the hypersurface $Z$ is locally defined by $y_1=0$ and
$$\omega=d x_1\wedge\frac{d y_1}{y_1}+\sum_{i=2}^n d x_i\wedge d y_i.$$
\end{theorem}

\begin{proof}
Let $\omega=\alpha\wedge \frac{d f}{f}+\beta$, and $\tilde{\alpha}=i^*\alpha$ and $\tilde{\beta}=i^*\beta$, where $i:Z\hookrightarrow M$ is the inclusion. As seen in Proposition \ref{thetamu}, for all $p\in Z$ we have $\tilde{\alpha}_p$ nonvanishing, $\tilde{\alpha}_p\wedge\tilde{\beta}_p\neq0$ and $\tilde{\beta}_p\in\Lambda^2( T_p^*Z)$ of rank $n-1$.
Hence, by Remark \ref{remark} we can assume,

$$\omega|_Z=(d x_1\wedge\frac{d y_1}{y_1}+\sum_{i=2}^n d x_i\wedge d y_i)|_Z.$$
The desired result now follows from Theorem \ref{Moser}.
\end{proof}

\subsection{Global Moser theorem for $b$-symplectic manifolds}

When $M$ is compact, we obtain a global result:

\begin{theorem} \label{globalmoser}\textbf{[$b$-Moser theorem]}
Suppose that $M$ is compact and let $\omega_0$ and $\omega_1$ be two $b$-symplectic forms on $(M,Z)$. Suppose that $\omega_t$, for $0\leq t\leq 1$, is a smooth family of $b$-symplectic forms on $(M,Z)$ joining $\omega_0$ and $\omega_1$ and such that the $b$-cohomology class $[\omega_t]$ does not depend on $t$. Then, there exists a family of diffeomorphisms $\gamma_t:M\to M$, for $0\leq t\leq 1$ such that $\gamma_t$ leaves $Z$ invariant and $\gamma_t^*\omega_t=\omega_0$.
\end{theorem}

\begin{proof}
As in the proof of Theorem \ref{Moser}, the existence of the desired isotopy $\gamma_t$ relies on the existence of a smooth family of $b$-vector fields $v_t$ such that
\begin{equation}\label{moremosertrick}
\cL_{v_t}\omega_t=\frac{d\omega_t}{dt}.
\end{equation}
Integrating this family of $b$-vector fields will then yield a suitable $\gamma_t$.

Because  $[\omega_t]$ is independent of $t$, we have $[\frac{d\omega_t}{dt}]=\frac{d[\omega_t]}{dt}=0$, and so there exists a family $\mu_t\in\,^b\Omega^1(M)$ such that $\frac{d\omega_t}{dt}=d \mu_t$. Equation (\ref{moremosertrick}) becomes
$$\iota_{v_t}\omega_t=\mu_t$$

\noindent where $\omega_t$ is a symplectic $b$-form, and therefore defines an isomorphism between
$b$-forms and $b$-vector fields. Therefore, the vector field $v_t$ defined by the equation above is
a $b$-vector field and hence tangent to $Z$. The flow integrating $v_t$, $\gamma_t$ gives the desired diffeomorphism $\gamma_t:M\to M$, leaving $Z$ invariant (since $v_t$ is tangent to $Z$) and $\gamma_t^*\omega_t=\omega_0$.
\end{proof}

\begin{remark}
Using the isomorphism of Theorem \ref{thm:poissoncohomology}, one can rewrite the Moser theorems in this section in the language of Poisson cohomology. We chose not to include it here because we find that this language does not simplify statements or proofs.

Another possible restatement can be obtained combining Theorems \ref{globalmoser} and \ref{thm:mazzeomelrose} ($b$-Mazzeo-Melrose) to replace the condition on $[\omega_t]$ by conditions on the corresponding families of elements of $H^2(M)$ and $H^1(Z)$.
\end{remark}

\subsection{Revisiting Radko's classification theorem}

In the classical symplectic world, the Moser theorem applied to the two-dimensional case tells us that surfaces with the same symplectic volume must be symplectomorphic. The $b$-version of the Moser theorem, Theorem \ref{globalmoser} yields as its two-dimensional instance the classification of Radko surfaces.

Let $M$ be a compact two-dimensional manifold and $\Pi_0$ and $\Pi_1$ be two Radko surface structures on it which vanish along the same set of curves $\gamma_1,\ldots,\gamma_n$, induce the same modular periods along these curves and have the same regularized Liouville volume. Let us prove, as a corollary of Theorem \ref{globalmoser} that the respective dual $b$-symplectic forms $\omega_0$ and $\omega_1$ on $(M,Z=\bigcup_{i=1}^n \gamma_i)$ are $b$-symplectomorphic.

For $j=0,1$, write $\omega_j=\alpha_j\wedge\frac{d f_j}{f_j}+\beta_j$. We can assume that $f_0$ and $f_1$ have the same sign, since $\frac{d(-f_i)}{-f_i}=\frac{df_i}{f_i}$, and then using a diffeomorphism, that $f_0=f_1=f$. Also, since $M$ is 2-dimensional, $\beta_j$ must be of the form $\beta_j=\mu_j\wedge df$, for some one-form $\mu_j$, and so by renaming $\alpha_j+f\mu_j$ to $\alpha_j$, which note doesn't change its restriction to $Z$, we can write
$$\omega_j=\alpha_j\wedge\frac{df}{f}.$$
Let $\omega_t= (1-t)\omega_0 + t\omega_1=\alpha_t\wedge\frac{df}{f}$, where $\alpha_t=(1-t)\alpha_0+t\alpha_1$. We want to show that the $b$-cohomology class $[\omega_t]$ are both independent of $t$, and that for all $0\leq t\leq 1$, the form $\omega_t$ is $b$-symplectic.

The modular vector field is unique up to adding hamiltonian vector fields, which are tangent to the symplectic leaves. In dimension two, because the leaves contained in $Z$ are points, the modular vector field is unique. Furthermore, a vector field on a closed simple curve is uniquely determined (up to diffeomorphism) by its period, and so the periods of the modular vector field along the curves $\gamma_1,\ldots,\gamma_n$ completely determine the restriction of the modular vector field to $Z$. Thus, $v_{\text{mod}\, 0}|_Z=v_{\text{mod} \,1}|_Z$. The pullback $\tilde{\alpha_j}=i^*\alpha_j$ is characterized by $\tilde{\alpha_j}(v_{\text{mod}\, j}|_Z)=1$, so we have $\tilde{\alpha_0}=\tilde{\alpha_1}$, and thus $\omega_0|_Z=\omega_1|_Z$.
By Proposition \ref{honest}, the two-form $(\omega_0-\omega_1)$ is an honest de Rham form. The regularized Liouville volumes associated with $\omega_0$ and $\omega_1$ being equal means that
$$\int_M (\omega_0-\omega_1)=0$$
both as a $b$-integral (see Remark \ref{integrate}) and -- since it is an honest de Rham form that we are integrating -- as an honest integral. Thus, the (honest de Rham) cohomology class of $[\omega_0-\omega_1]$ is zero, and so is its $b$-cohomology class. Therefore, $[\omega_t]=[\omega_0]=[\omega_1]\in\, ^bH^2(M)$.

Lastly, we must show that $\omega_t$ is a $b$-symplectic form for all $t$. Away from $Z$, the forms $\omega_0$ and $\omega_1$ are honest area forms inducing the same orientation (the orientation is induced by the modular vector field's restriction to $Z$), and on  $Z$ the form $\tilde{\alpha_t}=\tilde{\alpha_0}=\tilde{\alpha_1}$ does not vanish, so $\omega_t=\alpha_t\wedge\frac{df}{f}$ is nondegenerate and hence $b$-symplectic.

\section{Invariants associated to the exceptional hypersurface of a $b$-symplectic manifold}\label{section:invariantsassociatedtoabpoissonstructure}

Given a $b$-symplectic structure, the exceptional hypersurface is endowed with a Poisson structure which is regular of corank one.
In this section we study and we completely characterize the foliation invariants of the codimension one symplectic foliation on $Z$.
These invariants of $Z$  were introduced in \cite{guimipi} in the context of codimension one symplectic foliation. The definition of these invariants does not require the existence of a $b$-symplectic structure on $(M,Z)$ or even that of a manifold $M$ of which $Z$ is a hypersurface. The first invariant requires a transversally orientable codimension-one regular foliation on $Z$, the second one a Poisson structure on $Z$ which induces a symplectic foliation with those characteristics. When $Z$ is the exceptional hypersurface of a $b$-symplectic manifold $(M,Z)$, we get for free such a foliation and Poisson structure on $Z$.

\subsection{The defining one-form of a foliation and the first obstruction class}

Let $Z$ be an odd-dimensional manifold, and $\mathcal{F}$ a transversally orientable codimension-one foliation on $Z$.

\begin{definition}
A form $\alpha\in\Omega^1(Z)$ is a \textbf{defining one-form} of the foliation $\mathcal{F}$ if it is nowhere vanishing and $i^*_L\alpha=0$ for all leaves $L\stackrel{i_L}{\hookrightarrow}Z$.
\end{definition}

In particular, when the foliation $\mathcal{F}$ on $Z$ is the symplectic foliation induced on $Z$ by a $b$-symplectic structure on $(M,Z)$, a defining one-form can be chosen such that $\alpha(v_{mod}|_Z)=1$. With this extra condition, the defining one-form is unique even when we consider a different volume form on $M$: this causes the modular vector field $v_{mod}$ to change by a Hamiltonian vector field, which is tangent to the leaves of $\mathcal{F}$. Also in this particular case, the modular one-form $\alpha$ will necessarily be closed:

\begin{proposition}
If $(M,Z)$ is a $b$-symplectic manifold then a defining one-form $\alpha$ of the symplectic foliation on $Z$  that satisfies $\alpha(v_{mod}|_Z)=1$ is necessarily closed.
\end{proposition}

\begin{proof}

The flow of $v_{mod}$ preserves the foliation $\mathcal{F}$, so it also preserves  corresponding
defining one-form $\alpha$: $\cL_{v_{mod}|_Z}\alpha=0$. Applying Cartan's formula we have
$$0=\cL_{v_{mod}|_Z}\alpha=d\iota_{v_{mod}|_Z}\alpha+\iota_{v_{mod}|_Z} d\alpha.$$
The first summand vanishes because $\alpha({v_{mod}|_Z})=1$, and so the second summand vanishes as well: $(d\alpha)({v_{mod}|_Z},-)=0$.

Since for a point $p\in Z$ we have the decomposition
$$T_p Z=\text{span}({v_{mod}|_Z})\oplus T_p L,$$
where $L$ is the corank 2 symplectic leaf through $p$, it only remains to check that $d\alpha(v_1,v_2)=0$ for $v_1,v_2 \in T_p L.$ But $di_L^*\alpha=i_L^*d\alpha$ and therefore $d\alpha(v_1,v_2)=0$.
\end{proof}

\begin{remark}
Alternatively, observe that the defining one-form $\tilde{\alpha}$  defined in Proposition \ref{thetamu} satisfies the condition $\tilde{\alpha}(v_{mod}|_Z)=1$  as we saw in Remark \ref{rem:remark}. Also observe that $\tilde{\alpha}$ is closed because of Proposition \ref{thetamu}.
\end{remark}

We follow \cite{guimipi} for the definition the first obstruction class, an invariant related to the defining one-form of a foliation $\mathcal{F}$ on $Z$. Consider the short exact sequence of complexes
$$0\longrightarrow\alpha\wedge\Omega(Z)\stackrel{i}{\longrightarrow}\Omega(Z)\stackrel{j}{\longrightarrow}\Omega(Z)/\alpha\wedge\Omega(Z)\longrightarrow0.$$
Because $i^*_L d\alpha=0$ for all $L\in\mathcal{F}$ and $\alpha$ is a defining one-form of the foliation $\mathcal{F}$, we have
$$d\alpha=\beta\wedge\alpha\,\,\text{for some}\,\,\beta\in\Omega^1(Z).$$
Furthermore, because $0=d(d\alpha)=d\beta\wedge\alpha-\beta\wedge\beta\wedge\alpha=d\beta\wedge\alpha,$ the two-form $d\beta$ is in the subcomplex $\alpha\wedge\Omega(Z)$, and so $d(j\beta)=0$.

\begin{definition}
The \textbf{first obstruction class} of the foliation $\mathcal{F}$ is $$c_\mathcal{F}=\left[j\beta\right]\in H^1(\Omega(Z)/\alpha\wedge\Omega(Z)).$$
\end{definition}

Note that since a different defining one-form will be $\alpha'=f\alpha$ with $f$ a nonvanishing function, the complexes $\alpha\wedge\Omega(Z)$ and $\Omega(Z)/\alpha\wedge\Omega(Z)$ do not depend on the choice of $\alpha$. Furthermore, one can prove that the class $c_\mathcal{F}$ is independent of the choice of $\alpha$ and that:

\begin{theorem}\cite{guimipi}
The first obstruction class vanishes, $c_\mathcal{F}=0$, if and only if one can choose the defining one-form $\alpha$ of $\mathcal{F}$ to be closed.
\end{theorem}

As proved in  \cite{guimipi}, manifolds with vanishing first obstruction class are unimodular (vanishing modular class). In particular, the first obstruction class of the foliation induced on $Z$ by a $b$-symplectic structure on $(M,Z)$ vanishes.

\subsection{The defining two-form of a foliation and the second obstruction class}

Now assume that $Z$ is endowed with a regular corank one Poisson structure $\Pi$, that $\mathcal{F}$ is the corresponding foliation of $Z$ by symplectic leaves and that $c_\mathcal{F}=0$. Fix a closed defining one-form $\alpha$.

\begin{definition}
A form $\omega\in\Omega^2(Z)$ is a \textbf{defining two-form} of the foliation $\mathcal{F}$ induced by the Poisson structure $\Pi$ if $i^*_L\omega$ is the symplectic form induced by $\Pi$ on each leaf $L\stackrel{i_L}{\hookrightarrow}Z$.
\end{definition}

When the Poisson structure $\Pi$ on $Z$ is induced by a $b$-symplectic structure on $(M,Z)$, a defining two-form can be chosen such that $\iota_{v_{mod}|_Z}\omega=0$.

\begin{proposition}
If $(M,Z)$ is a $b$-symplectic manifold then a defining two-form $\omega$ of the symplectic foliation $\mathcal{F}$ on $Z$  that satisfies $\iota_{v_{mod}|_Z}\omega=0$ is necessarily closed.
\end{proposition}

\begin{proof}
The flow of $v_{mod}$ preserves the Poisson structure $\Pi$, so it also preserves the corresponding  defining two-form $\omega$: $\cL_{v_{mod}|_Z}\omega=0$. Applying Cartan's formula we have
$$0=\cL_{v_{mod}|_Z}\omega=d\iota_{v_{mod}|_Z}\omega+\iota_{v_{mod}|_Z} d\omega.$$
The first summand vanishes by hypothesis, and so the second summand vanishes as well: $\iota_{v_{mod}|_Z} d\omega=0$.

Since for $p\in Z$ and $L$ the corank 2 symplectic leaf through $p$ we have the decomposition
$$T_p Z=\text{span}({v_{mod}|_Z})\oplus T_p L,$$
it only remains to check that $d\omega(v_1,v_2,v_3)=0$ for $v_1, v_2, v_3 \in T_p L.$ But $i_L^*\omega$ is a symplectic form on $L$, hence closed, and therefore $d\omega(v_1,v_2,v_3)=0$.
\end{proof}

Considering a different volume form on $M$ changes the modular vector field from $v_\text{mod}$ to $v_\text{mod}'=v_\text{mod}+u_f$, where $u_g$ is the hamiltonian vector field of the function $g$, and the defining two-form from $\omega$ to $\omega'=\omega+dg\wedge\alpha=\omega+d(g\wedge\alpha)$, thus not changing the cohomology class $[\omega']=[\omega]\in H^2(Z)$.

We follow \cite{guimipi} for the definition of the second obstruction class, an invariant related to the defining two-form of the symplectic foliation $\mathcal{F}$ induced by a regular corank one Poisson structure $\Pi$ on $Z$. Because $i^*_L d\omega=d(i^*_L\omega)=0$ for all $L\in\mathcal{F}$ and $\alpha$ is a defining one-form of the foliation $\mathcal{F}$, we have $$d\omega=\mu\wedge\alpha\,\,\text{for some}\,\,\mu\in\Omega^2(Z).$$
Furthermore, because $0=d(d\omega)=d\mu\wedge\alpha-\mu\wedge d\alpha$ and $\alpha$ is closed, the 3-form $d\mu$ is in the subcomplex $\alpha\wedge\Omega(Z)$, and so $d(j\mu)=0$.

\begin{definition}
The \textbf{second obstruction class} of the foliation $\mathcal{F}$ is $$\sigma_\mathcal{F}=\left[j\mu\right]\in H^2(\Omega(Z)/\alpha\wedge\Omega(Z)).$$
\end{definition}

\noindent One can prove that the class $\sigma_\mathcal{F}$ is independent of the choice of $\omega$.

\begin{theorem}\cite{guimipi}
The second obstruction class vanishes, $\sigma_\mathcal{F}=0$, if and only if one can choose the defining two-form $\omega$ of $\mathcal{F}$ to be closed.
\end{theorem}

In particular, the second obstruction class (and also the first) of the foliation induced on $Z$ by a $b$-symplectic structure on $(M,Z)$ vanishes.

This second invariant has also been studied by Gotay \cite{gotay} in the setting of coisotropic
 embeddings and has a nice interpretation for symplectic fiber bundles (see
 \cite{weinsteinetal}).

Finally, we recall a result from \cite{guimipi} that gives a topological description of a compact Poisson manifold whose regular corank one symplectic foliation contains a compact leaf and which has vanishing first and second invariants. Indeed, in that paper we proved that if the invariants $c_\mathcal{F}$ and $\sigma_\mathcal{F}$ vanish, then the exceptional hypersurface is a symplectic mapping torus if one of the leaves is compact \cite[Theorem 19]{guimipi}. In this section we have proved the converse, thus establishing the following theorem:

\begin{theorem}\label{thm: semilocalstructure}
If $Z$ is an oriented compact connected regular Poisson manifold of corank one and $\mathcal{F}$ is its symplectic foliation, then $c_\mathcal{F} = \sigma_\mathcal{F}=0$ if and only if there exists a Poisson vector field transversal to $\mathcal{F}$. If furthermore $\mathcal{F}$ contains a compact leaf $L$, then every leaf of $\mathcal{F}$ is symplectomorphic to $L$, and $Z$ is the total space of a fibration over $\mathbb{S}^1$ and the mapping torus \footnote{The mapping torus of $\phi:L\to L$ is the space $\frac{L\times\left[0,1\right]}{(x,0)\sim(\phi(x),1)}$.} of the symplectomorphism $\phi:L\to L$ given by the holonomy map of the fibration over $\mathbb{S}^1$.
\end{theorem}

\section{The extension problem}\label{section:theextensionproblem}

A $b$-symplectic structure on $(M,Z)$ induces on $Z$ a regular corank one Poisson structure and corresponding foliation by symplectic leaves. Conversely, we can ask: when is a manifold $Z$, endowed with a regular corank one Poisson structure $\Pi$ and corresponding symplectic foliation $\mathcal{F}$, the exceptional hypersurface of a $b$-symplectic manifold $(M,Z)$? And to what extent is such an extension unique? We will see that an extension of $Z$ to a $b$-symplectic tubular neighborhood exists exactly when the obstructions classes $c_\mathcal{F}$ and $\sigma_\mathcal{F}$ of $Z$ vanish, and that uniqueness is related to the modular vector class.

 Consider  $M=Z\times(-\varepsilon,\varepsilon)$ a tubular neighborhood of $Z$ and let $p:M\to Z$ be the projection onto the first factor. Because $Z$ is a Poisson submanifold, given a Poisson vector field $X$ on $M$, the vector field $p_*(X)$ is a Poisson vector field. Let us denote by $[p_*]$ the mapping
$[p_*]:H^1_{Poisson}(M)\longrightarrow H^1_{Poisson}(Z).$
Since the modular vector field of a $b$-Poisson structure is tangent to $Z$, the mapping $p_*$ sends the modular vector field (which is a Poisson vector field) at points of $Z$ to itself.
\begin{theorem}\label{thm:extension}
Let $\Pi$ be a regular corank one Poisson structure on a compact manifold $Z$, and $\mathcal{F}$ the induced foliation by symplectic leaves.

Then $c_\mathcal{F}=\sigma_\mathcal{F}=0$ if and only if $Z$ is the exceptional hypersurface of a $b$-symplectic manifold $(M,Z)$ whose $b$-symplectic form induces on $Z$ the Poisson structure $\Pi$.

Furthermore, two such extensions $(M_0,Z)$ and $(M_1,Z)$ are $b$-sym\-plec\-to\-mor\-phic on a tubular neighborhood of $Z$ if and only if the image of their modular   classes under the map  $[p_*]$ is the same.
\end{theorem}

\begin{proof}
We have seen in Section \ref{section:invariantsassociatedtoabpoissonstructure} that if $Z$ is the exceptional hypersurface of a $b$-symplectic manifold $(M,Z)$, then the obstructions $c_\mathcal{F}$ and $\sigma_\mathcal{F}$ vanish.

Conversely, assume that $c_\mathcal{F}=\sigma_\mathcal{F}=0$ and let $\alpha$ and $\omega$ be closed defining one- and two-forms of the foliation $\mathcal{F}$ on $Z$. Let $M=Z\times(-\varepsilon,\varepsilon)$ and $p:M\to Z$ be the projection onto the first factor. Define, with $t\in(-\varepsilon,\varepsilon)$,
\begin{equation}\label{eq:normalform}\tilde{\omega}=p^*\alpha\wedge\frac{dt}{t}+p^*\omega. \end{equation}
By construction, $\tilde\omega$ is a non-degenerate $b$-form which induces on $Z\times\{0\}$ the given Poisson structure $\Pi$. Furthermore, it is closed because $\alpha$ and $\omega$ are closed, and so $\tilde\omega$ is a $b$-symplectic form.

Now, let $(M_0,Z)$ and $(M_1,Z)$ be two such extensions, with $b$-symplectic forms $\tilde{\omega}_0$ and $\tilde{\omega}_1$ respectively. By choosing small enough tubular neighborhoods of $Z$ in $M_0$ and $M_1$ we can assume that we have two $b$-symplectic structures $\tilde{\omega}_0$ and $\tilde{\omega}_1$ on the same $b$-manifold $(U,Z)$. By hypothesis, these induce on $Z$ the same Poisson structure $\Pi$, and their corresponding modular vector fields $v_{\text{mod}\,0}$ and $v_{\text{mod}\,1}$ (for some choice of volume form) are such that $v_{\text{mod}\,0}|_Z$ differs from  $v_{\text{mod}\,1}|_Z$ by a hamiltonian vector field on $Z$. Then, by Theorem \ref{theorem:relativemoser2}, the two structures $\tilde\omega_0$ and $\tilde\omega_1$ are $b$-symplectomorphic in a possibly smaller tubular neighborhood of $Z$.
\end{proof}

\begin{example}
Let $Z=\mathbb S^3$ and $\mathcal F$ be a codimension-one foliation (for example the Reeb foliation). If the first obstruction class $c_\mathcal F$ were to vanish, there would exist a closed defining one-form $\alpha$. But a closed one-form on $\mathbb S^3$ is necessarily exact, $\alpha=df$, and since $\mathbb S^3$ is compact $\alpha$ would vanish at the singular points of $f$. Thus, $\mathbb S^3$ cannot be the exceptional hypersurface of a $b$-symplectic manifold.
\end{example}

\begin{example} Consider $Z=\mathbb T^3$ with coordinates $\theta_1,\theta_2,\theta_3$ and $\mathcal{F}$ the codimension-one foliation on $Z$ with leaves the different $k$-levels, $k\in\mathbb{R}$, of
$$\theta_3=a\theta_1+b \theta_2+k,$$
where $a,b\in\RR$ are fixed and independent over $\mathbb{Q}$; then each leaf is diffeomorphic to $\RR^2$ \cite{mamaev}. The one-form
$$\alpha=\frac{a}{a^2+b^2+1}\,d\theta_1+\frac{b}{a^2+b^2+1}\,d\theta_2-\frac{1}{a^2+b^2+1}\,d\theta_3$$
is a defining one-form for $\mathcal{F}$ and there is a Poisson structure $\Pi$ on $Z$  which induces the foliation $\mathcal{F}$ and for which
$$\omega=d\theta_1\wedge d\theta_2+b\, d\theta_1\wedge d\theta_3-a\,d\theta_2\wedge d\theta_3$$
is the defining two-form. Both $\alpha$ and $\omega$ are
closed, and so the invariants $c_\mathcal{F}$ and $\sigma_\mathcal{F}$ vanish.

The manifold $Z$ is the exceptional hypersurface of  a $b$-symplectic tubular neighborhood $(\mathcal{U},Z)$ with $b$-symplectic form
$$\tilde\omega= \frac{dt}{t}\wedge p^*\alpha+ p^*\omega,$$
where $p:\mathcal{U}=Z\times(-\varepsilon,\varepsilon)\to Z$ is projection and $t\in(-\varepsilon,\varepsilon)$. On $Z$, the $b$-symplectic form $\tilde{\omega}$ induces the Poisson structure $\Pi$.

To produce a global example of a compact $b$-symplectic extension of $Z$, we use instead $Z=\mathbb T^3\times\{0,\pi\}$. Then $(\mathbb T^4,Z)$ with
$$\tilde\omega= \frac{1}{\sin\theta_4}\,d\theta_4 \wedge p^*\alpha + p^*\omega$$ is a compact $b$-symplectic manifold.

\end{example}

\begin{remark}
As an application of the normal form formula given by (\ref{eq:normalform}) in a neighborhood of $Z$, Marcut and Osorno obtain a refinement of  some  of the $b$-cohomology constraints that we have obtained in Section \ref{section:cohomology} (see \cite{marcutosorno1} for more details).
\end{remark}


\section{Integrability of $b$-symplectic manifolds as Poisson manifolds}\label{section:integrability}

We begin reviewing the notion of symplectic groupoid as defined by Weinstein in \cite{weinstein3}. We recall that a groupoid is a set $\Gamma$ equipped with a subset $\Gamma_0$ of identity elements, a source map $\alpha:\Gamma\to\Gamma_0$, a target map $\beta:\Gamma\to\Gamma_0$, a multiplication operation $(x,y)\to x\cdot y$ on the set of $(x,y)$'s satisfying $\beta(x)=\alpha(y)$, and an inversion operation $\iota:\Gamma\to\Gamma$, with the data above satisfying the obvious generalizations of the usual group axioms.

If $\Gamma$ is a $C^\infty$ manifold and the data above is also $C^\infty$, then $\Gamma$ is called a differentiable groupoid. If $\Gamma$ is also equipped with a symplectic form $\Omega$ for which the manifold $\Gamma_3=\{(z,x,y): z=x\cdot y\}$ is Lagrangian in $(\Gamma,\Omega)\times(\Gamma,-\Omega)\times(\Gamma,-\Omega)$ then $\Gamma$ is called a symplectic groupoid.

It is easy to see from the definition above that $\Gamma_0$ is a Lagrangian submanifold of $\Gamma$ and with a little more effort one can show that $\Gamma_0$ has an intrinsic Poisson structure for which the maps $\alpha$ and $\beta$ are Poisson maps. Thus for every symplectic groupoid $\Gamma$ one gets an intrinsically associated Poisson manifold $\Gamma_0$, and as pointed out in \cite{weinstein3} it is natural to ask if the converse is true: Given a Poisson manifold $\Gamma_0$ can it be ``integrated'' to a symplectic groupoid $\Gamma$ having $\Gamma_0$ as its set of identity elements? Counterexamples show that this isn't always the case, but that for many familiar examples of Poisson manifolds it holds true. Two examples, which will figure in our application of this theory to $b$-symplectic manifolds, are the following:

\begin{enumerate}
\item If $\Gamma_0=M$ with $(M,\omega)$ a symplectic manifold, then one can take $\Gamma$ to be the pair groupoid $M\times M^-$ with the composition law $(x,y)\cdot(y,z)=(x,z)$.
\item Let $G$ be a Lie group, $\mathfrak{g}$ its Lie algebra and $\Gamma_0=\mathfrak{g}^*$ with its Poisson bivector field
$$f\in\mathfrak{g}^*\to\Pi_f\in\Lambda^2(T_f\mathfrak{g}^*)\cong\Lambda^2(\mathfrak{g})^*,$$
where $\Pi_f(x\wedge y)=<f,[x,y]>.$ In this case, the symplectic groupoid integrating $\Gamma_0$ is $T^*G$ with source and target maps
$$\alpha,\beta:T^*G\to G\times\mathfrak{g}^*\to\mathfrak{g}^*$$
given by the right and left trivializations on $T^*G$.
\end{enumerate}

In the 1990's, the question posed by Weinstein was formulated in a more general context -- ``Given a Lie algebroid, can it be integrated to a Lie groupoid?'' -- and in 2003, necessary and sufficient conditions for this to be true were given by Crainic and Fernandes in \cite{marui2}. In particular, they showed that a Poisson manifold $M$ can be integrated to a symplectic groupoid if its Poisson bivector field
$$m\in M\to\Pi_m\in\Lambda^2(T_m M)$$
is non-degenerate on an open dense subset of $M$. Because a $b$-Poisson bivector field $\Pi$ is non-degenerate on $M\setminus Z$, all $b$-symplectic manifolds are integrable.

Nonetheless, it would be nice to have a concrete description of the $\Gamma$ corresponding to the $\Gamma_0$ given by a $b$-symplectic manifold $(M,\omega)$. Locally such a description can be given in terms of the two examples described earlier due to the fact that the $b$-symplectic form can be written in a neighborhood of a point $p\in Z$ as (see Remark \ref{remark:darbouxpoisson})
$$\omega=\omega_{L_p}+(\Pi^T)^\sharp.$$

\begin{proposition}
The {\bf{local symplectic groupoid}} integrating $M$ in a neighborhood of a point $p\in Z$ is the product groupoid
$$T^*G\times\Gamma,$$
where $G$ is the ``$ax+b$ group'' (with Lie algebra structure given by $[e_1,e_2]=e_2$) and $\Gamma$ is the pair groupoid $(L_p,\omega_{L_p})\times(L_p,\omega_{L_p}^-)$.
\end{proposition}

A semiglobal integrability result, i.e., a description of the symplectic groupoid integrating $M$ in a neighborhood the exceptional hypersurface $Z$, rather than the neighborhood of a point in $Z$, requires first that we discuss the integrability of $Z$. This was done in \cite{guimipi}, and we recall here the main results as applied to the case in study:

\begin{proposition}\cite{guimipi}
If $(M,Z)$ is a $b$-symplectic manifold, then the induced regular Poisson manifold $Z$ is integrable.
\end{proposition}

This was obtained using \cite[Corollary 14]{marui1} and as a consequence of the existence of a leafwise symplectic embedding for $Z$. This is guaranteed by the vanishing of the second invariant $\sigma_\mathcal{F}$, which, as recalled in section \ref{section:invariantsassociatedtoabpoissonstructure}, can be interepreted via Gotay's embedding theorem exactly as measuring the obstruction of the existence of a closed two-form on $Z$ which restricts to a symplectic form on each leaf of the foliation.

\begin{theorem} \label{thm:integratinggroupoidoftheexceptionalsurface}\cite{guimipi}
If $(M,Z)$ is a $b$-symplectic manifold and the induced symplectic foliation of $Z$ contains a compact leaf, then all leaves are compact and the hypersurface $Z$ is a mapping torus, and furthermore $Z$ is integrable and its Weinstein's symplectic groupoid is the associated mapping torus groupoid.
\end{theorem}

We describe succintly how the Weinstein's symplectic groupoid for $Z$ is constructed in \cite{guimipi}. For a point $p\in Z$ contained in the symplectic leaf $L_p$, consider $\Sigma=((L_p\times L_p^{-})\times T^*(\mathbb R),\Omega+d\lambda_{\text{liouville}})$ with the groupoid product structure. Then, consider a Poisson vector field on $Z$ which is generates the $\mathbb{S}^1$-action on the mapping torus $Z$, in our case we may consider $v_{mod}|_Z$, and use a lift of this vector field to construct the associated mapping torus groupoid. This will be give us the Weinstein's symplectic groupoid of $Z$.

A semiglobal integrability result for $b$-symplectic manifolds can be obtained by combining Theorem \ref{thm:integratinggroupoidoftheexceptionalsurface} with the semiglobal model for $b$-sympletic manifolds given by Theorem \ref{thm: semilocalstructure}. For this we consider the projection map $p:\cU\to Z$ from a tubular neighborhood $\cU$ of $Z$ and use it to pullback the groupoid structure on $Z$ to $\cU$, which is fairly simple when $Z$ is a mapping torus, as described above.

The description and classification of symplectic groupoids of various Poisson manifolds and in particular of $b$-symplectic manifolds is recent work of Gualtieri and Li. We refer the reader to their preprint \cite{Gualtierili} for further details.

\section{Integrable systems on $b$-symplectic manifolds}\label{section:integrablesystems}

This paper is the second of a {series of papers} on $b$-symplectic geometry. In the first paper we studied the geometry of regular Poisson manifolds with codimension-one symplectic foliations, and in particular proved the results about such manifolds that are quoted in Section \ref{section:invariantsassociatedtoabpoissonstructure}. In {a subsequent} paper in this series we will study integrable systems and Hamiltonian actions on $b$-symplectic manifolds.

We start with the following definition,

\begin{definition} An {\bf{adapted integrable system}} on a $b$-symplectic manifold $(M^{2n}, Z, \omega)$ is a collection of $n$ smooth functions $\{f_1, \dots, f_n\}$ called \textbf{first integrals} such that $f_1$ is a defining function for $Z$, $\{f_i, f_j\} = 0$ for all $i, j$, the functions are functionally independent (i.e., $df_1 \wedge \dots df_n  \neq 0$) on a dense set, and the restrictions of $\{f_2, \dots, f_n\}$ to each symplectic leaf of $Z$ are functionally independent on the leaf. The function ${\bf{F}}=(f_1,\dots, f_n)$ is the \textbf{moment map} of the integrable system.
\end{definition}

 The following example is the canonical local model of a regular adapted integrable system on a $b$-symplectic manifold.
\begin{example}
Consider $\mathbb R^{2n}$ with coordinates $(x_1,y_1,\dots, x_{n-1},y_{n-1}, t,z)$ and $b$-symplectic form
$\omega=\sum_{i=1}^{n-1} dx_i\wedge dy_i+\frac{1}{t}\,dt\wedge dz$. The bivector field corresponding to $\omega$ is
\[
\Pi=\sum_{i=1}^{n-1}\frac{\partial}{\partial x_i}\wedge \frac{\partial}{\partial y_i}+t\frac{\partial}{\partial t}\wedge \frac{\partial}{\partial z}.
\]
The functions $\{x_1, \dots, x_{n-1}, t\}$ define an adapted integrable system with moment map $\mathbf{F}=(x_1,\dots, x_{n-1},t)$.
\end{example}

Most of the classical normal form and action-angle results for symplectic manifolds also hold for $b$-symplectic manifolds in a neighborhood of points in $Z$ including possible non-degenerate singularities as studied by Eliasson, Zung and the second author of this paper (\cite{eliassonthesis,eliassonelliptic,evathesis,mirandazungequiv}). In particular we will see:

\begin{theorem}
Given an adapted integrable system with non-degenerate singularities on a
$b$-symplectic manifold $(M,Z)$, there exist Eliasson-type normal forms
in a neighborhood of points in $Z$ and the minimal rank for these
singularities is one along $Z$.
\end{theorem}


\end{document}